\documentclass{article}

\usepackage{arxiv}

\usepackage[utf8]{inputenc} % allow utf-8 input
\usepackage[T1]{fontenc}    % use 8-bit T1 fonts
\usepackage{hyperref}       % hyperlinks
\usepackage{url}            % simple URL typesetting
\usepackage{booktabs}       % professional-quality tables
\usepackage{amsfonts}       % blackboard math symbols
\usepackage{nicefrac}       % compact symbols for 1/2, etc.
\usepackage{microtype}      % microtypography
\usepackage{lipsum}		% Can be removed after putting your text content
\usepackage{graphicx}
\usepackage{doi}

\usepackage{amsthm}
\usepackage{mathtools}
\usepackage{amsmath}
\usepackage{amstext}
\usepackage{amssymb}
\usepackage{amsfonts}
\usepackage{physics}
\usepackage{caption}
\usepackage{graphicx}
\usepackage{bbm}
\usepackage{mathrsfs}
\usepackage{upgreek}
\usepackage{multirow}
\usepackage{placeins}
\usepackage{hyperref}
\usepackage{bm}
\usepackage{nicefrac}
\usepackage{enumitem}
\usepackage{xcolor}
\usepackage{wasysym}
\usepackage{pifont}
\usepackage{fancyhdr}
\usepackage{amsbsy}
\usepackage{stmaryrd}
\usepackage{cases}
\usepackage{wasysym}
\usepackage{appendix}
\usepackage{sidecap}
\usepackage{graphicx}
\usepackage{float}
\usepackage{dsfont}
\usepackage{indentfirst}
\usepackage{mathrsfs}
\usepackage{listings}
\usepackage{empheq,etoolbox}

\newcommand{\A}{\mathcal{A}}

\renewcommand{\phi}{\varphi}
\newcommand{\K}{\mathcal{K}}
\newcommand{\B}{\mathcal{B}}

\DeclarePairedDelimiter{\norma}{\lVert}{\rVert}

\newcommand{\Stot}{\mathcal{S}}
\newcommand{\Sint}{\mathcal{S}^I}

\newcommand{\SB}{\mathcal{S}^{\Gamma}}

\theoremstyle{definition}

\newtheorem{Definition}{Definition}[section]
\newtheorem{Theorem}[Definition]{Theorem}

\newtheorem{Lemma}[Definition]{Lemma}
\newtheorem{remark}[Definition]{Remark}

\title{CVEM-BEM coupling with decoupled orders for 2D exterior Poisson problems}

\author{{\small Luca Desiderio} \\
{\small	Dipartimento di Scienze Matematiche, Fisiche e Informatiche}\\
{\small	Universit{\`a} di Parma}\\
{\small 	Parma, 43124, Italy} \\
{\small	\texttt{luca.desiderio@unipr.it}} \\
	\And
{\small	Silvia Falletta} \\
{\small	Dipartimento di Scienze Matematiche ``G.L. Lagrange''} \\ 
{\small	Politecnico di Torino} \\
{\small	Torino, 10129, Italy} \\
{\small	\texttt{silvia.falletta@polito.it}} \\
	\And
{\small	Matteo Ferrari} \\
{\small	Dipartimento di Scienze Matematiche ``G.L. Lagrange'' }\\ 
{\small	Politecnico di Torino} \\
{\small	Torino, 10129, Italy} \\
{\small	\texttt{matteo.ferrari@polito.it}} \\
	\And
{\small	Letizia Scuderi} \\
{\small	Dipartimento di Scienze Matematiche ``G.L. Lagrange''} \\ 
{\small	Politecnico di Torino} \\
{\small	Torino, 10129, Italy} \\
{\small	\texttt{letizia.scuderi@polito.it}} \\
}

\renewcommand{\headeright}{}
\renewcommand{\undertitle}{}
\renewcommand{\shorttitle}{}

\hypersetup{
pdftitle={A template for the arxiv style},
pdfsubject={q-bio.NC, q-bio.QM},
pdfauthor={David S.~Hippocampus, Elias D.~Striatum},
pdfkeywords={First keyword, Second keyword, More},
}

\begin{document}
\maketitle

\begin{abstract}
For the solution of 2D exterior Dirichlet Poisson problems we propose the coupling of a Curved Virtual Element Method (CVEM) with a Boundary Element Method (BEM), by using decoupled approximation orders. We provide optimal convergence error estimates, in the energy and in the weaker $\textit{L}^\text{2}$-norm, in which the CVEM and BEM contributions to the error are separated. This allows taking advantage of the high order flexibility of the CVEM to retrieve an accurate discrete solution by using a low order BEM.
 The numerical results confirm the a priori estimates and show the effectiveness of the proposed approach.
\end{abstract}


\section{Introduction}\label{sec_1_introduction}
In this paper we deal with the following 2D problem

\begin{equation}\label{dirichlet_problem}
\left\{
\begin{aligned}
	&-\Delta u_{e}(\mathbf{x}) = f(\mathbf{x}) & \mathbf{x}\in \Omega_{e},\\
	&u_{e}(\mathbf{x})=0		&	\mathbf{x}\in \Gamma_0,
	\end{aligned}
		\right.
\end{equation}
where 
$\Omega_{e}:=\mathbf{R}^{2}\setminus\overline{\Omega}_{\tiny{0}}$ is an unbounded domain, exterior to an open bounded one $\Omega_{\tiny{0}}$, with Lipschitz boundary $\Gamma_{\tiny{0}}$. It is known (see \cite{JohnsonNedelec1980} and the references therein) that Problem \eqref{dirichlet_problem} admits a unique solution in the space 
\[
W^1(\Omega_e) := \left\{ v :  \omega v \in L^2(\Omega_e), \nabla v \in [L^2(\Omega_e)]^2\right\}
\]
with $\omega(\mathbf{x}) := \left(\sqrt{1+\|\mathbf{x}\|^2}\left(1+\log\left(\sqrt{1+\|\mathbf{x}\|^2}\right)\right)\right)^{-1}$, satisfying the asymptotic conditions
\begin{equation}\label{dirichlet_problem_3}
u_e(\mathbf{x}) = \alpha + O\left(\frac{1}{\|\mathbf{x}\|}\right) \quad \text{and} \quad \nabla u_e(\mathbf{x})  = O\left(\frac{1}{\|\mathbf{x}\|^2}\right) \quad \text{for} \quad \|\mathbf{x}\| \rightarrow \infty.
\end{equation}
The constant $\alpha$ represents the asymptotic behaviour of $u_e$ at infinity and, here, its value is not fixed in advance.

The above problem is of interest
in many engineering and physical applications, for example when studying electric and thermal plane fields on infinite domains produced by point sources, or when solving problems of fluid flows around obstacles.
Many and various numerical methods have been proposed and analysed for its solution, among which we mention the traditional BEM. This latter is the most natural way to deal with unbounded domains (for a reference, see \cite{SauterSchwab2011} and the bibliography therein contained). Another common approach is the coupling of a classical variational or finite difference method with a transparent (absorbing or non-reflecting) condition defined on an artificial boundary $\Gamma$, properly chosen to delimit a finite computational domain. Among the most commonly used Non-Reflecting Boundary Conditions (NRBCs), those of integral type are exact (i.e. not approximated) and allow treating artificial boundaries of arbitrary, even non-convex, shapes.

The aim of this paper is to propose such a coupling by means of the interior CVEM and the one-equation BEM. Standard VEMs have been applied to a wide variety of interior problems (see the pioneering \cite{BeiraoBrezziCangianiManziniMariniRusso2013} for the Poisson problem and \cite{AntoniettiManziniVerani2020, MascottoPerugiaPichler2019, ArtioliMarfiaSacco2020, BeiraoCanutoNochettoVacca2021} for more recent applications), but only few papers deal with exterior problems (see \cite{GaticaMeddahi2019, GaticaMeddahi2020,  DesiderioFallettaScuderi2021, DesiderioFallettaFerrariScuderi2021} for elliptic equations).
Among the CVEM approaches till now investigated, we mention those proposed in \cite{BeiraoRussoVacca2019} and \cite{BeiraoBrezziMariniRusso2020}. Although the latter deals with local polynomial preserving VEM spaces, we choose the former since it is well-suited for problems characterized by computational domains with prescribed curved boundaries, like ours.

The choice of using VEM, or the more general CVEM, is mainly motivated by the following reasons: it allows us to consider meshes whose elements can be of general shape, and to use local discrete spaces of arbitrarily high order by maintaining the simplicity of implementation independent of it. Moreover, the nature of the VEM allows decoupling the approximation orders and the mesh grids associated with the domain and boundary methods, without the need of using special auxiliary variables (like mortar ones) for the coupling. Indeed, by exploiting the peculiar construction of the VEM, it is possible to add hanging nodes on the edges of the elements that belong to the artificial boundary, without significantly modifying the structure of the interior mesh.

For what concerns the one-equation BEM, we recall that it has been proposed in the well known
Johnson \& N\'ed\'elec Coupling (JNC) (see \cite{BrezziJohnson1979,JohnsonNedelec1980}) and it is based on a single Boundary Integral Equation (BIE) that involves the integral operators associated to the fundamental solution (and its normal derivative) of the Laplace equation.

In the recent work \cite{GaticaMeddahi2019}, in which a similar problem has been studied, the authors consider the Costabel \& Han Coupling (CHC) (see  \cite{Costabel1987, Han1990}) combined with an interior VEM. 
This approach yields to a symmetric and non-positive definite scheme but, involving a BIE of hypersingular type, turns out to be quite onerous from the computational point of view. Even if the CHC has been applied in several contexts, the JNC turns out to be very appealing from the engineering point of view, this latter being cheaper and easier to implement. We remark in addition that, unlikely in \cite{GaticaMeddahi2019}, we deal with the asymptotic condition \eqref{dirichlet_problem_3} that entails  \begin{equation}\label{eq:identity}
\int_{\Gamma}\lambda_e({\mathbf{y}}) \mathrm{d} \Gamma_{{\mathbf{y}}} = 0,
\end{equation}
where $\lambda_e({\mathbf{y}}):=\frac{\partial u_e}{\partial\mathbf{n}}(\mathbf{y})$ denotes the normal derivative of $u_e$ along the artificial boundary $\Gamma$. As a consequence, suitable spaces satisfying identity \eqref{eq:identity} have to be considered.    

For the discretization of our coupled problem we consider a full Galerkin approach based on a CVEM in the interior of the computational domain and on a BEM associated to basis functions chosen in such a way that \eqref{eq:identity} is satisfied.
 We study the proposed approach from the theoretical point of view in a quite general framework, and we provide optimal error estimates in the energy and in the weaker $L^2$-norm. 
 In particular, since we consider here curved domains, the use of curvilinear elements instead of polygonal ones, allows us to reach the optimal convergence rate for degrees of accuracy higher than 2, avoiding the sub-optimal rate caused by the approximation of the domain.
 
 By a careful study, we show that the source of the approximation error of the discrete solution, both in the energy and in the $L^2$-norm, can be split into two contributions: a boundary (BEM) and an interior (CVEM) one. In particular, we show that the boundary contribution behaves like $h_\partial^{k_\partial}$ ($h_\partial$ denoting the maximum edge length of the artificial boundary and $k_\partial$ representing the BEM polynomial degree of accuracy), and the interior one like $h_\circ^{k_\circ}$ ($h_\circ$ being the element diameter and $k_\circ$ the CVEM order degree). Hence, for $h_\partial \ll h_\circ$ and by fixing $k_\partial$, it results  that the bulk error dominates the boundary one up to a certain CVEM order, an aspect that allows obtaining a high accuracy of the global scheme with a low BEM order.

\
The paper is organized as follows: in the next section we present the model problem for the Poisson equation and its reformulation
in a bounded region, obtained by introducing the artificial boundary and its associated one equation Boundary Integral Non Reflecting Boundary Condition (BI-NRBC). In Section \ref{sec_3_weak_pb} we introduce the variational formulation of the problem restricted to the finite computational domain. In Section \ref{sec_4_galerkin_pb} we apply the Galerkin method and we prove error estimates in the energy and in the $L^2$-norm in an abstract framework, provided that suitable hypotheses are assumed. 
Then we show that these latter are satisfied by the CVEM-BEM approximation spaces introduced in Section \ref{sec_5_discrete_scheme}. Finally, in the last section we detail the choice of the particular basis functions used for the approximation of the normal derivative unknown, and we present some numerical test which confirm the theoretical results.

\section{The model problem}\label{sec_2_model_pb}
Let $\Omega_{e}:=\mathbf{R}^{2}\setminus\overline{\Omega}_{\tiny{0}}$ be an unbounded domain, exterior to an open bounded domain $\Omega_{\tiny{0}}\subset\mathbf{R}^{2}$, and denote by $\Gamma_{\tiny{0}}:=\partial {\Omega}_{e}$ its Lipschitz boundary having positive Haussdorf measure (see Figure \ref{fig:rob_domain} (a)).
We consider the exterior Dirichlet Poisson problem \eqref{dirichlet_problem}
in the unknown solution $u_{e}$, where $f\in L^{2}(\Omega_{e})$ represents a source term having a compact support in $\Omega_e$.

To determine the solution $u_{e}$ of Problem (\ref{dirichlet_problem}) by means of an interior domain method, we surround the physical obstacle $\Omega_0$ by an artificial boundary $\Gamma$; this allows decomposing $\Omega_{e}$ into a finite computational domain $\Omega$, bounded internally by $\Gamma_0$ and externally by $\Gamma$, and an infinite residual one, denoted by $\Omega_{\infty}$ (see Figure \ref{fig:rob_domain} (b)). For the theoretical analysis of the numerical approach we propose, we need to assume that $\Gamma_0$ consists of a finite number of curves of class $C^{m+1}$, with $m\geq0$, and that $\Gamma$ is a contour of class $C^\infty$.
\begin{figure}[H]
	\centering
	\includegraphics[width=0.65\textwidth]{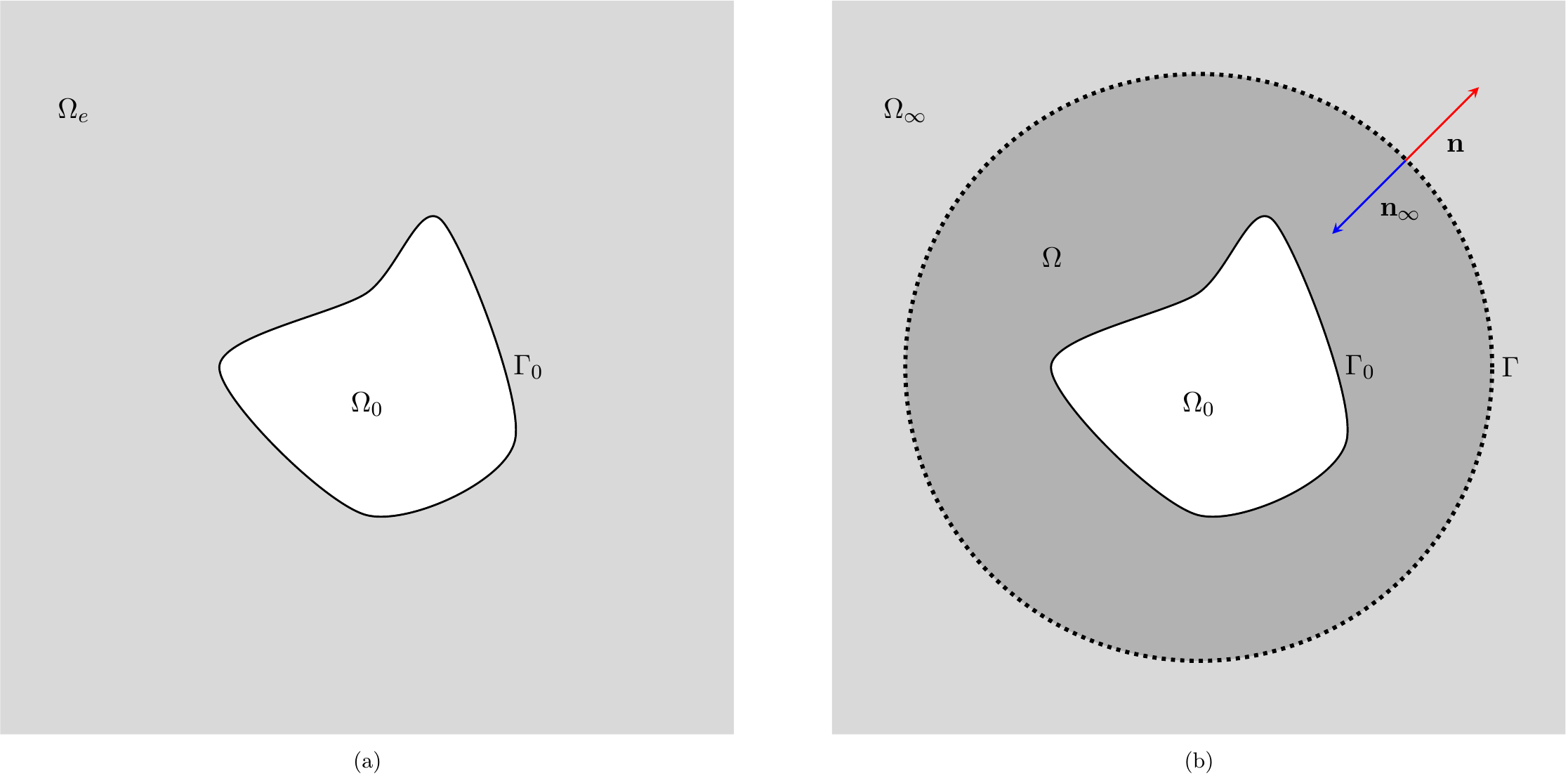}
	\caption{Model problem setting.}
	\label{fig:rob_domain}
\end{figure}
\noindent 
Denoting by $u$ and $u_{\infty}$ the restrictions of the solution $u_{e}$ to $\Omega$ and $\Omega_{\infty}$ respectively, and by $\mathbf{n}$ and $\mathbf{n}_{\infty}$ the unit normal vectors on $\Gamma$ pointing outside $\Omega$ and $\Omega_{\infty}$ (consequently $\mathbf{n}_{\infty}=-\mathbf{n}$), we consider the following compatibility and equilibrium conditions on $\Gamma$:
\begin{equation}
\label{compatibility_condition}
u(\mathbf{x}) =  u_{\infty}(\mathbf{x}), \qquad        \frac{\partial u}{\partial{\mathbf{n}}}(\mathbf{x})=-\frac{\partial u_{\infty}}{\partial{\mathbf{n}_{\infty}}}(\mathbf{x}),	\qquad \mathbf{x}\in \Gamma.
\end{equation}

In the above relations and in the sequel we omit, for simplicity, the use of the trace operators to indicate the restriction of $H^1$ functions to the boundary $\Gamma$ from the exterior or interior.

Assuming, for simplicity, that $\Gamma$ is chosen such that $\text{supp}(f)$ is a bounded subset of $\Omega$, the following Kirchhoff's formula 
\begin{equation}\label{boundary_integral_representation}
u_{\infty}(\mathbf{x})=\int_{\Gamma}G(\mathbf{x},\mathbf{y})\frac{\partial u_{\infty}}{\partial\mathbf{n}_{\infty}}(\mathbf{y})\,\dd\Gamma_{\mathbf{y}}-\int_{\Gamma}\frac{\partial G}{\partial\mathbf{n}_{\infty,\mathbf{y}}}(\mathbf{x},\mathbf{y})u_{\infty}(\mathbf{y})\,\dd\Gamma_{\mathbf{y}} + \alpha, \quad \mathbf{x}\in\Omega_{\infty}\setminus\Gamma,
\end{equation}
allows us to represent the solution $u_{\infty}$ in $\Omega_{\infty}$. In \eqref{boundary_integral_representation},  $G$ and $\partial G/\partial\mathbf{n}_{\infty,\mathbf{y}}$  denote, respectively, the fundamental solution of the 2D Laplace equation and its normal derivative with respect to the unit vector $\mathbf{n}_{\infty,\mathbf{y}}$ having initial point in $\mathbf{y} \in\Gamma$. Their expression is given by
\begin{equation*}\label{kernels}
G(\mathbf{x},\mathbf{y})=-\frac{1}{2\pi}\log r \quad \text{and} \quad \frac{\partial G}{\partial\mathbf{n}_{\infty,\mathbf{y}}}(\mathbf{x},\mathbf{y})=\frac{1}{2\pi}\frac{\mathbf{r}\cdot\mathbf{n}_{\infty,\mathbf{y}}}{r^{2}},
\end{equation*}
where $r=\| \mathbf{r} \|=\|\mathbf{x}-\mathbf{y}\|$.
It is known that the trace of \eqref{boundary_integral_representation} on $\Gamma$ reads
\begin{equation}\label{boundary_integral_equation_operators}
\frac{1}{2}u_{\infty}(\mathbf{x})-\text{V}\frac{\partial u_{\infty}}{\partial\mathbf{n}_{\infty}}({\mathbf{x}})-\text{K}u_{\infty}(\mathbf{x}) - \alpha=0, \qquad \mathbf{x}\in\Gamma,
\end{equation}
where $\text{V} \colon H^{-\nicefrac{1}{2}}(\Gamma)\to H^{\nicefrac{1}{2}}(\Gamma)$ and $\text{K}\colon H^{\nicefrac{1}{2}}(\Gamma)\to H^{\nicefrac{1}{2}}(\Gamma)$ represent, respectively, the continuous (see \cite{HsiaoWendland2008}) single- and double-layer integral operators, defined by 
\begin{equation*}\label{single_layer_operator}
\text{V}\psi(\mathbf{x}):=\int_{\Gamma}G(\mathbf{x},\mathbf{y})\psi(\mathbf{y})\,\dd\Gamma_{\mathbf{y}}, \qquad \mathbf{x}\in\Gamma
\end{equation*}
and 
\begin{equation*}\label{single_layer_operator}
\text{K}\phi(\mathbf{x}):=-\int_{\Gamma}\frac{\partial G}{\partial\mathbf{n}_{\infty,\mathbf{y}}}(\mathbf{x},\mathbf{y})\phi(\mathbf{y})\,\dd\Gamma_{\mathbf{y}}, \qquad \mathbf{x}\in\Gamma.
\end{equation*}
\noindent
To determine the solution of Problem \eqref{dirichlet_problem} in the finite computational domain $\Omega$, we impose \eqref{boundary_integral_equation_operators} as BI-NRBC on $\Gamma$.
In particular, introducing the additional unknown $\displaystyle \lambda({\mathbf{y}}):=\frac{\partial u}{\partial\mathbf{n}}(\mathbf{y})$ and taking into account (\ref{compatibility_condition}), the new problem defined in $\Omega$ takes the form:
\begin{subequations}\label{model_problem_full}
	\begin{empheq}[left=\empheqlbrace]{align}
	\label{model_problem_full_1} &-\Delta u(\mathbf{x})=f(\mathbf{x})&  &\mathbf{x}\in \Omega\\
&u(\mathbf{x})=0& 						 &\mathbf{x}\in \Gamma_0\\
	\label{model_problem_full_3} &\frac{1}{2}u(\mathbf{x})+\text{V}\lambda({\mathbf{x}})-\text{K}u(\mathbf{x})- \alpha=0& &\mathbf{x}\in\Gamma.
	\end{empheq}
\end{subequations}
\noindent
We point out that the asymptotic conditions \eqref{dirichlet_problem_3} coupled with \eqref{model_problem_full_3} imply that $\langle\lambda,1\rangle_{\Gamma} = 0$,
where $\langle \cdot,\cdot \rangle_{\Gamma}$ denotes the duality pairing between  $H^{-\nicefrac{1}{2}}(\Gamma)$ and $H^{\nicefrac{1}{2}}(\Gamma)$.
This justifies the introduction of the space $H_{0}^{-\nicefrac{1}{2}}(\Gamma):=\left\{\lambda\in H^{-\nicefrac{1}{2}}(\Gamma) \ : \ \langle \lambda,1\rangle_{\Gamma}=0\right\}$ in which we will look for the unknown $\lambda$.

\section{The variational formulation}\label{sec_3_weak_pb}

Let us introduce the bilinear form  $a:H^{1}(\Omega)\times H^{1}(\Omega)\rightarrow\mathbf{R}$
\begin{equation*}
a(u,v):=\int\limits_{\Omega}\nabla u(\mathbf{x})\cdot\nabla v(\mathbf{x})\,\dd\mathbf{x}.
\end{equation*}
The variational formulation of Problem (\ref{model_problem_full}) consists in finding $u\in H_{0,\Gamma_0}^{1}(\Omega):= \{ u \in H^{1}(\Omega) : u = 0 \, \text{ on } \Gamma_0\}$ and $\lambda\in H_0^{-\nicefrac{1}{2}}(\Gamma)$ such that 
\small
\begin{subequations}\label{model_problem_variational}
	\begin{empheq}[left=\empheqlbrace]{align}
	\label{model_problem_variational_1} &a(u,v)-\langle \lambda,v\rangle_{\Gamma}=(f,v)_{L^2(\Omega)}&  &\forall\, v\in H^{1}_{0,\Gamma_{0}}(\Omega),\\
	\label{model_problem_variational_2} &\langle \mu,\left(\frac{1}{2}\text{I}-\text{K}\right)u\rangle_{\Gamma}+\langle \mu,\text{V}\lambda\rangle_{\Gamma}=0 & 						 &\forall\, \mu\in H_0^{-\nicefrac{1}{2}}(\Gamma),
	\end{empheq}
\end{subequations}
\normalsize
\noindent
where $\text{I}$ stands for the identity operator and $(\cdot,\cdot)_{L^2(\Omega)}$ denotes the $L^{2}(\Omega)$-inner product.
It is worth noting that, since we test Equation \eqref{model_problem_full_3} with $\mu\in H_0^{-\nicefrac{1}{2}}(\Gamma)$, satisfying by definition $\langle \mu,1\rangle_{\Gamma}=0$, the
unknown constant $\alpha$ does not appear in the variational formulation \eqref{model_problem_variational}. Nevertheless, the asymptotic
behaviour $\alpha$ is intrinsic to the interior domain problem, and it can be recovered by the numerical scheme when choosing $\Gamma$ sufficiently far from the obstacle (see Example \ref{example_2}).

 To reformulate the above problem in operator form, following \cite{JohnsonNedelec1980}, we introduce the Hilbert space $V:=H^{1}_{0,\Gamma_{0}}(\Omega)\times H_0^{-\nicefrac{1}{2}}(\Gamma)$, equipped with the norm
\begin{equation*}\label{norm_V}
\left\|\hat{u}\right\|_{V}^{2}:=\left\|u\right\|_{H^{1}(\Omega)}^{2} + \left\|\lambda\right\|_{H^{-\nicefrac{1}{2}}(\Gamma)}^{2}, \quad \text{for} \ \hat u = (u,\lambda).
\end{equation*}
Then, we define the bilinear form $\mathcal{A}:V\times V\rightarrow\mathbf{R}$ 
\begin{equation*}
\mathcal{A}(\hat{u},\hat{v}):=a(u,v)-\langle \lambda,v\rangle_{\Gamma}+\langle \mu,u\rangle_{\Gamma}+2\langle \mu,\text{V}\lambda\rangle_{\Gamma}-2\langle \mu,\text{K}u\rangle_{\Gamma},
\end{equation*}
for $\hat u = (u,\lambda)$ and $\hat v = (v,\mu)$, and the linear continuous operator $\mathcal{L}_{f}:V\rightarrow\mathbf{R}$ 
\begin{equation*}\label{def_lf_kappa}
\mathcal{L}_{f}(\hat{v}):=(f,v)_{L^2(\Omega)}.
\end{equation*}
\noindent
Hence, we rewrite Problem (\ref{model_problem_variational}) as follows: find $\hat{u}\in V$ such that
\begin{equation}\label{model_problem_operator}
\mathcal{A}(\hat{u},\hat{v})=\mathcal{L}_{f}(\hat{v}) \qquad \forall\,\hat{v}\in V,
\end{equation}
\noindent
whose well-posedness has been proved in \cite{JohnsonNedelec1980} (see Lemma 2).
\\

\noindent
Finally, for the forthcoming analysis, it is convenient to rewrite $\mathcal{A}=\mathcal{B}+\mathcal{K}$ where the 
the bilinear forms
$\mathcal{B},\mathcal{K}:V\times V\rightarrow\mathbf{R}$ are defined as follows:
\begin{equation}\label{definition_B_K}
 \mathcal{B}(\hat{u},\hat{v}) := a(u,v)-\langle \lambda,v\rangle_{\Gamma}+\langle \mu,u\rangle_{\Gamma}+2\langle \mu,\text{V}\lambda\rangle_{\Gamma}, \,\,\,
 \mathcal{K}(\hat{u},\hat{v}) := -2\langle \mu,\text{K}u\rangle_{\Gamma}. 
\end{equation}
\noindent

In the following sections, for the solution of Problem \eqref{model_problem_operator}, we will describe a numerical approach consisting of a CVEM-BEM coupling. This method and the corresponding theoretical analysis is based on that proposed  for the Helmholtz problem in \cite{DesiderioFallettaFerrariScuderi2021}, to which we refer whenever the theoretical results therein proved hold in our context as well. It is worth noting that the theoretical analysis for the exterior Poisson problem cannot be obtained as a particular sub-case of the Helmholtz one given in \cite{DesiderioFallettaFerrariScuderi2021}, by simply choosing the wave number equal to zero. Indeed, the NRBC associated to the Laplace equation is different from that of the Helmholtz one, both for what concerns the kernel functions appearing in the boundary integral operators  and the choice of the discrete function spaces for the approximation of the unknown $\lambda$. In fact, in this case, since we do not know a priori the asymptotic value $\alpha$ in \eqref{model_problem_full_3}, the choice of the space $H_0^{-\nicefrac{1}{2}}(\Gamma)$ becomes mandatory and, as a consequence, a proper discrete subspace of it is needed. Moreover,
another important novelty of the theoretical study, with respect to that of \cite{DesiderioFallettaFerrariScuderi2021}, consists in the use of decoupled degrees of approximation for the interior CVEM and the BEM. This allows in particular the application of the CVEM with order higher than that of the BEM, a key aspect for the global scheme since the BEM requires high efforts to efficiently compute the associated system matrices. 

\section{The numerical method}\label{sec_4_galerkin_pb}

To describe the numerical approach we propose to solve \eqref{model_problem_operator}, we start by introducing a suitable decomposition of the domain $\Omega$, which consists of generic elements and is not limited to the more commonly used triangles.

Let us denote by $E$ a generic ``polygon'' having at most one curved edge and by $h_E$ its diameter; similarly we denote by $e$ a generic ``edge'', eventually curved, and by $h_e$ its length. We introduce a sequence $\{\mathcal{T}_{h_\circ}\}_{h_\circ}$ of unstructured meshes $\mathcal{T}_{h_\circ}=\left\{E\right\}$, which cover the domain $\Omega$, where $h_\circ:=\max_{E\in\mathcal{T}_{h_\circ}}h_E$. We denote by $\mathcal{T}_{h_\partial}^{\Gamma}$ the decomposition of the artificial boundary $\Gamma$ which, according to the regularity assumption required for $\Gamma$, consists of curvilinear parts.
 The subscript $h_\partial$ denotes the mesh size defined by $h_\partial = \max_{e \in \mathcal{T}^{\Gamma}_{h_\partial}}h_e$.

We suppose there exists a constant $\varrho>0$ such that for each $h_\circ$ and for each element $E\in\mathcal{T}_{h_\circ}$, $E$ is star-shaped with respect to a ball of radius greater than $\varrho h_{E}$ and the length of any (eventually curved) edge of $E$ is greater than $\varrho h_{E}$.

For any $k\in\mathbf{N}$, let $P_{k}(E)$ be the space of polynomials of degree $k$ defined on $E$, and $\Pi_{k}^{\nabla,E}:H^{1}(E)\rightarrow P_{k}(E)$ be the local polynomial $H^{1}$-projection, defined such that for $v\in H^{1}(E)$:
\begin{equation*}
\begin{cases}
\displaystyle\int_{E}\nabla\Pi_{k}^{\nabla,E}v \cdot \nabla q\,\dd E = \displaystyle\int_{E}\nabla v \cdot\nabla q\,\dd E \qquad  \forall\, q\in P_{k}(E),\\[10pt]
\displaystyle\int_{\partial E}\Pi_k^{\nabla,E} v\,\dd s = \displaystyle\int_{\partial E}v\,\dd s.
\end{cases}
\end{equation*}
The local projector $\Pi_{k}^{\nabla,E}$ can be naturally extended to the global one $\Pi_{k}^{\nabla}:H^{1}(\Omega)\rightarrow P_{k}(\mathcal{T}_{h_\circ})$ as follows:
\begin{equation*}
\left(\Pi_{k}^{\nabla} v\right)_{\vert_E}:=\Pi_{k}^{\nabla,E}v_{\vert_E} \quad\forall\, v\in H^{1}(\Omega),
\end{equation*}
$P_k(\mathcal{T}_{h_\circ})$ being the space of piecewise polynomials with respect to the decomposition $\mathcal{T}_{h_\circ}$ of $\Omega$. 
Moreover, let $\Pi_{k}^{0,E}:L^{2}(E)\rightarrow P_{k}(E)$ be the local polynomial $L^{2}$-projection operator, defined such that for $v\in L^{2}(E)$ 
\begin{equation*}
\int_{E}\Pi_{k}^{0,E}v \, q\,\dd E= \int_{E} v \, q\,\dd E \qquad \forall\, q\in P_{k}(E).
\end{equation*}
By introducing the local bilinear form $a^{\text{\tiny{E}}}:H^{1}(E)\times H^{1}(E)\rightarrow\mathbf{R}$ given by
\begin{align} \label{ae}
a^{\text{\tiny{E}}}(u,v):=\int\limits_{E}\nabla u(\mathbf{x})\cdot\nabla v(\mathbf{x})\,\dd\mathbf{x},
\end{align}
we can write $a(u,v)=\sum\limits_{E\in\mathcal{T}_{\boldsymbol{h}}}a^{\text{\tiny{E}}}(u,v)$.

Finally, we introduce the product space $H^1(\mathcal{T}_{h_\circ}) := \prod_{E\in\mathcal{T}_{h_\circ}} H^1(E)$ and define the associated broken $H^1$-norm:
$$\| v \|_{H^1(\mathcal{T}_{h_\circ})}^2 := \sum_{E\in \mathcal{T}_{h_\circ}} \| v \|^2_{H^1(E)}.$$
To apply the Galerkin method to Problem (\ref{model_problem_operator}), we introduce the discrete spaces $Q_{h_\circ}^{k_\circ}\subset H_{0,\Gamma_{0}}^{1}(\Omega)$ and $X_{h_\partial}^{k_\partial}\subset H_0^{-\nicefrac{1}{2}}(\Gamma)$ associated to the meshes $\mathcal{T}_{h_\circ}$ and $\mathcal{T}_{h_\partial}^{\Gamma}$, respectively, and the product space  $V_{\boldsymbol{h}}^{\boldsymbol{k}}:=Q_{h_\circ}^{k_\circ}\times X_{h_\partial}^{k_\partial}$. Then, the Galerkin method consists in  
finding $\hat{u}_{\boldsymbol{h}}\in V_{\boldsymbol{h}}^{\boldsymbol{k}}$ such that
\begin{equation}\label{model_problem_galerkin}
\mathcal{A}_{\boldsymbol{h}}(\hat{u}_{\boldsymbol{h}},\hat{v}_{\boldsymbol{h}}):=\mathcal{B}_{\boldsymbol{h}}(\hat{u}_{\boldsymbol{h}},\hat{v}_{\boldsymbol{h}})+\mathcal{K}(\hat{u}_{\boldsymbol{h}},\hat{v}_{\boldsymbol{h}})=\mathcal{L}_{f,\boldsymbol{h}}(\hat{v}_{\boldsymbol{h}}) \quad \forall\,\hat{v}_{\boldsymbol{h}}\in V_{\boldsymbol{h}}^{\boldsymbol{k}},
\end{equation}
\noindent
where $\mathcal{A}_{\boldsymbol{h}},\mathcal{B}_{\boldsymbol{h}}:V_{\boldsymbol{h}}^{\boldsymbol{k}}\times V_{\boldsymbol{h}}^{\boldsymbol{k}}\rightarrow\mathbf{R}$ and $\mathcal{L}_{f,\boldsymbol{h}}:V_{\boldsymbol{h}}^{\boldsymbol{k}}\rightarrow\mathbf{R}$ are suitable approximations of $\mathcal{A},\mathcal{B}$ and $\mathcal{L}_{f}$, respectively. 

Proceeding analogously as in \cite{DesiderioFallettaFerrariScuderi2021}, we introduce  sufficient conditions on the discrete spaces, on the bilinear form $\mathcal{B}_{\boldsymbol{h}}$ and on the linear operator $\mathcal{L}_{f,\boldsymbol{h}}$ to guarantee existence and uniqueness of the solution $\hat{u}_{\boldsymbol{h}}\in V_{\boldsymbol{h}}^{\boldsymbol{k}}$ and to prove convergence error estimates.

In particular we assume: for any $s \geq 1$

\begin{enumerate}[label=(H1.\alph*), ref=H1.\alph*]
	{\setlength\itemindent{5pt} \item\label{H1.a}  approximation in $Q_{h_\circ}^{k_\circ}$: for all $v \in H^{s+1}(\Omega)$\\ $\underset{v_{h_\circ}\in Q_{h_\circ}^{k_\circ}}{\text{inf}} \left\|v-v_{h_\circ}\right\|_{H^{1}(\Omega)} \apprle h_\circ^{\text{min}(s,k_\circ)} \left\|v\right\|_{H^{s+1}(\Omega)}$}; 	{\setlength\itemindent{5pt} \item\label{H1.b} approximation in $X_{h_\partial}^{k_\partial}$: for all $\mu \in H^{s-\nicefrac{1}{2}}(\Gamma) \cap H_0^{-\nicefrac{1}{2}}(\Gamma)$}\\ $\underset{\mu_{h_\partial}\in  X_{h_\partial}^{k_\partial}}{\text{inf}} \left\|\mu-\mu_{h_\partial}\right\|_{H^{-\nicefrac{1}{2}}(\Gamma)} \apprle h_\partial^{\text{min}(s,k_\partial)} \left\|\mu \right\|_{H^{s-\nicefrac{1}{2}}(\Gamma)}.$ 
\end{enumerate}

In the above assumptions the notation $Q_1 \apprle Q_2$ (as well as $Q_1 \apprge Q_2$ in what follows) means that the quantity $Q_1$ is bounded from above (resp. from below) by $c\,Q_2$, where $c$ is a positive constant that, unless explicitly stated, does not depend on any relevant parameter involved in the definition of $Q_1$ and $Q_2$.

According to the definition of the $\|\cdot\|_V$ norm, \eqref{H1.a} and \eqref{H1.b} ensure the following approximation property for the product space $V_{\boldsymbol{h}}^{\boldsymbol{k}}$: \\ 

for $s \ge 1$, given $\hat{v}=(v,\mu)\in H^{s+1}(\Omega)\times (H^{s-\nicefrac{1}{2}} (\Gamma) \cap H_0^{-\nicefrac{1}{2}}(\Gamma))$, there exists $\hat{v}_{\boldsymbol{h}}^{I}=(v_{h_\circ}^{I},\mu_{h_\partial}^{I})\in V_{\boldsymbol{h}}^{\boldsymbol{k}}$ such that
\begin{equation}\label{int_property_0} 
\left\|\hat{v} - \hat{v}_{\boldsymbol{h}}^{I}\right\|_V\apprle h_\circ^{\text{min}(s,k_\circ)} \left\|v\right\|_{H^{s+1}(\Omega)}+ h_\partial^{\text{min}(s,k_\partial)} \left\|\mu \right\|_{H^{s-\nicefrac{1}{2}}(\Gamma)}.
\end{equation}

Recalling that the evaluation of the bilinear form $\mathcal{B}$ on elements of $V_{\boldsymbol{h}}^{\boldsymbol{k}}$ is well defined provided that $a(\cdot,\cdot)$ is split into the sum of the local contributions $a^{\text{\tiny{E}}}(\cdot,\cdot)$, and assuming that the approximated bilinear form $\mathcal{B}_{\boldsymbol{h}}$ is well defined on the space $H^1(\mathcal{T}_{h_\circ})$, we further assume:

\begin{enumerate}[label=(H2.\alph*), ref=H2.\alph*]
	{\setlength\itemindent{10pt} \item\label{H2.a} $k_\circ$-consistency: for all $\hat{q}\in P_{k_\circ}(\mathcal{T}_{h_\circ})\times X_{h_\partial}^{k_\partial}$ and $\hat{v}_{\boldsymbol{h}}\in V_{\boldsymbol{h}}^{\boldsymbol{k}}$
		\begin{equation*}
		\mathcal{B}_{\boldsymbol{h}}(\hat{q},\hat{v}_{\boldsymbol{h}}) = \mathcal{B}(\hat{q},\hat{v}_{\boldsymbol{h}}), \qquad \mathcal{B}_{\boldsymbol{h}}(\hat v_{\boldsymbol{h}},\hat{q}) = \mathcal{B}(\hat v_{\boldsymbol{h}},\hat{q});
		\end{equation*}}
	{\setlength\itemindent{10pt} \item\label{H2.b} continuity: for all $\hat{v}_{\boldsymbol{h}},\hat{w}_{\boldsymbol{h}}\in V_{\boldsymbol{h}}^{\boldsymbol{k}}$
		\begin{equation*}	
		 \abs*{\B_{\boldsymbol{h}}(\hat{v}_{\boldsymbol{h}},\hat{w}_{\boldsymbol{h}})}\apprle\left\|\hat{v}_{\boldsymbol{h}}\right\|_V \left\|\hat{w}_{\boldsymbol{h}}\right\|_V;
	 \end{equation*}}
	{\setlength\itemindent{10pt} \item\label{H2.c} ellipticity: for all $\hat{w}_{\boldsymbol{h}} \in V_{\boldsymbol{h}}^{\boldsymbol{k}}$
		\begin{equation*}
		\mathcal{B}_{\boldsymbol{h}}(\hat{w}_{\boldsymbol{h}},\hat{w}_{\boldsymbol{h}}) \apprge\left\|\hat{w}_{\boldsymbol{h}}\right\|_{V}^{2}.
		\end{equation*}}
\end{enumerate}

In the following theorem we show the validity of the inf-sup condition for the discrete bilinear form $\A_{\boldsymbol{h}}$.
\begin{Theorem}\label{JN4}
	Assuming \eqref{H1.a}, \eqref{H1.b} and  \eqref{H2.a}--\eqref{H2.c}, for $h_\circ$ and $h_\partial$ small enough, it holds that
	\begin{equation*}
	\sup_{\hat v_{\boldsymbol{h}}\in V_{\boldsymbol{h}}^{\boldsymbol{k}}, \hat v_{\boldsymbol{h}}\ne 0} \frac{\A_{\boldsymbol{h}}(\hat w_{\boldsymbol{h}} , \hat v_{\boldsymbol{h}})}{\norma{\hat v_{\boldsymbol{h}}}_V} \apprge \norma{\hat w_{\boldsymbol{h}}}_{V} \quad \forall \, \hat w_{\boldsymbol{h}} \in V_{\boldsymbol{h}}^{\boldsymbol{k}}.
	\end{equation*}
\end{Theorem}
\begin{proof}	
Following the proof of Lemma 4 in \cite{JohnsonNedelec1980}, it is possible to assert that, for any $\hat w_{\boldsymbol{h}} \in V_{\boldsymbol{h}}^{\boldsymbol{k}}$, there exists $\hat v_{\boldsymbol{h}} \in V_{\boldsymbol{h}}^{\boldsymbol{k}}$ such that
	\begin{equation} \label{eq:JN4.5}
	\norma{\hat v_{\boldsymbol{h}}}_V \apprle \norma{\hat w_{\boldsymbol{h}}}_V
	\end{equation} and
	\begin{equation} \label{eq:JN4.6}
	\A(\hat w_{\boldsymbol{h}},\hat v_{\boldsymbol{h}}) \apprge (1-h_\circ-h_\partial)\norma{\hat w_{\boldsymbol{h}}}^2_V.
	\end{equation}
	By exploiting the decoupled assumptions \eqref{H1.a} and \eqref{H1.b} in Lemma 4.5 of \cite{DesiderioFallettaFerrariScuderi2021}, we obtain that for $\hat v_{\boldsymbol{h}} = (v_{h_\circ},\mu_{h_\partial}) \in V_{\boldsymbol{h}}^{\boldsymbol{k}} \subset V$ there exists a unique $\hat v_{\boldsymbol{h}}^* = (v_{h_\circ}^*,\mu_{h_\partial}^*) \in V_{\boldsymbol{h}}^{\boldsymbol{k}}$ such that 
	\begin{equation}\label{eq:Boh_utile}
	\B_{\boldsymbol{h}}(\hat w_{\boldsymbol{h}}, \hat v_{\boldsymbol{h}}^*) = \B(\hat w_{\boldsymbol{h}}, \hat v_{\boldsymbol{h}}) \quad \forall \,\hat w_{\boldsymbol{h}} \in V_{\boldsymbol{h}}^{\boldsymbol{k}}
	\end{equation}
	and
	\begin{equation}\label{eq:JN4.7b}	\norma{\mu_{h_\partial}^* - \mu_{h_\partial}}_{H^{-\nicefrac{3}{2}}(\Gamma)}\apprle (h_\circ + h_\partial) \norma{\hat v_{\boldsymbol{h}}}_V.
	\end{equation}
	Recalling \eqref{definition_B_K} and \eqref{model_problem_galerkin}, and using \eqref{eq:Boh_utile}, we get:
	\begin{align*} 
	\A_{\boldsymbol{h}}(\hat w_{\boldsymbol{h}}, \hat v_{\boldsymbol{h}}^*) & = \B_{\boldsymbol{h}}(\hat w_{\boldsymbol{h}}, \hat v_{\boldsymbol{h}}^*) + \K(\hat w_{\boldsymbol{h}}, \hat v_{\boldsymbol{h}}^*) = \B(\hat w_{\boldsymbol{h}}, \hat v_{\boldsymbol{h}})+ \K(\hat w_{\boldsymbol{h}},\hat v_{\boldsymbol{h}}^*) 
	\\ &  =\A(\hat w_{\boldsymbol{h}}, \hat v_{\boldsymbol{h}})  + \K(\hat w_{\boldsymbol{h}},\hat v_{\boldsymbol{h}}^*) - \K(\hat w_{\boldsymbol{h}},\hat v_{\boldsymbol{h}})  \\
	&  = \A(\hat w_{\boldsymbol{h}}, \hat v_{\boldsymbol{h}}) - 2 \langle \mu_{h_\partial}^* - \mu_{h_\partial},\text{K}  w_{\boldsymbol{h}} \rangle_{\Gamma}.
	\end{align*}
	By applying the H{\"o}lder's inequality and using \eqref{eq:JN4.7b} we have
	\begin{align}  \nonumber
	\abs*{\langle \mu_{h_\partial}^* - \mu_{h_\partial},\text{K} w_{h_\circ} \rangle_{H^{-\nicefrac{3}{2}}(\Gamma) \times H^{\nicefrac{3}{2}}(\Gamma)}} & \apprle  \norma{\mu_{h_\partial}^* - \mu_{h_\partial}}_{H^{-\nicefrac{3}{2}}(\Gamma)} \norma{\text{K} w_{h_\circ}}_{H^{\nicefrac{3}{2}}(\Gamma)}  \\ & \label{eq:JN4.15} \apprle (h_\circ+h_\partial) \norma{\hat v_{\boldsymbol{h}}}_V \norma{\text{K} w_{h_\circ}}_{H^{\nicefrac{3}{2}}(\Gamma)}.
	\end{align}
	Then, using the continuity of $K : H^{\nicefrac{1}{2}}(\Gamma) \to H^{\nicefrac{3}{2}}(\Gamma)$ (see \cite{JohnsonNedelec1980}, formula (2.11)) and the trace theorem, we obtain
	\begin{equation*} %
	\norm{\text{K} w_{h_\circ}}_{H^{\nicefrac{3}{2}}(\Gamma)} \apprle \norm{ w_{h_\circ}}_{H^{\nicefrac{1}{2}}(\Gamma)} \apprle \norma{w_{h_\circ}}_{H^1(\Omega)} \le \norma{\hat w_{\boldsymbol{h}}}_V,
	\end{equation*}
	which, together with \eqref{eq:JN4.15}, implies
	\begin{equation} \label{eq:JN4.16}
	\abs*{\langle \mu_{h_\partial}^* - \mu_{h_\partial},\text{K} w_{h_\circ} \rangle_{H^{-\nicefrac{3}{2}}(\Gamma) \times H^{\nicefrac{3}{2}}(\Gamma)}} \apprle (h_\circ + h_\partial) \norma{\hat v_{\boldsymbol{h}}}_V \norma{\hat w_{\boldsymbol{h}}}_V.
	\end{equation}
	Finally, combining \eqref{eq:JN4.6}, \eqref{eq:JN4.16} and \eqref{eq:JN4.5}, we get
	\begin{align*}
	\A_{\boldsymbol{h}}(\hat w_{\boldsymbol{h}}, \hat v_{\boldsymbol{h}}^*) & \apprge (1-h_\circ - h_\partial)\norma{\hat w_{\boldsymbol{h}}}_V^2  - (h_\circ + h_\partial)\norma{\hat w_{\boldsymbol{h}}}_V\norma{\hat v_{\boldsymbol{h}}}_V \\
	& \apprge (1-h_\circ - h_\partial)\norma{\hat w_{\boldsymbol{h}}}_V^2,
	\end{align*}
	whence, for $h_\circ$ and $h_\partial$ small enough, the claim follows.
\end{proof}
\noindent
Theorem \ref{JN4} allows us to prove the following convergence error estimate in the $V$-norm for Problem \eqref{model_problem_galerkin}.
\begin{Theorem}
	\label{Theorem:JN1}
	Assume there exist $k_\circ,k_\partial$ such that for all $s \ge 1$, \eqref{H1.a},\eqref{H1.b}, and \eqref{H2.a}--\eqref{H2.c}  hold. Furthermore, assume that, for all $s \ge 1$, there exists $\sigma_s : L^2(\Omega) \to \mathbf{R}^+$ such that 
	 
	\begin{enumerate}[label=(H3.\alph*), ref=H3.\alph*]
		{\setlength\itemindent{10pt} \item\label{H3.a} for all $\hat{v}_{\boldsymbol{h}} \in V_{\boldsymbol{h}}^{\boldsymbol{k}}$ $$\abs*{\mathcal{L}_{f}(\hat{v}_{\boldsymbol{h}}) - \mathcal{L}_{f,\boldsymbol{h}}(\hat{v}_{\boldsymbol{h}})} \apprle \left(h_\circ^{\text{min}(s,k_\circ)} + h_\partial^{\text{min}(s,k_\partial)}\right)\left\|\hat{v}_{\boldsymbol{h}}\right\|_{V} \, \sigma_s(f).$$}
	\end{enumerate}

	Then, for $h_\circ$ and $h_\partial$ small enough, Problem \eqref{model_problem_galerkin} admits a unique solution $\hat u_{\boldsymbol{h}} \in V_{\boldsymbol{h}}^{\boldsymbol{k}}$ and if $\hat u = (u,\lambda)$, solution of Problem \eqref{model_problem_operator}, satisfies $\hat u \in H^{s+1}(\Omega) \times H^{s-\nicefrac{1}{2}}(\Gamma)$, the following estimate holds
	\begin{equation*}
	\norma{\hat u - \hat u_{\boldsymbol{h}}}_V \apprle \left(h_\circ^{\text{min}(s,k_\circ)}+h_\partial^{\text{min}(s,k_\partial)}\right) \left( \norma{u}_{H^{s+1}(\Omega)}  + \sigma_s(f)\right).
	\end{equation*}
\end{Theorem}
Since the proof of Theorem \ref{Theorem:JN1} can be obtained by proceeding analogously as in Theorem 4.8 of \cite{DesiderioFallettaFerrariScuderi2021}, for brevity we omit it here.

In what follows we provide the error estimate in the weaker $W$-norm, where $W := L^2(\Omega) \times H^{-\nicefrac{3}{2}}(\Gamma)$. To this aim, we start by assuming the following property:

\begin{enumerate}[label=({{H}}\text{3}.{{b}}{}), ref={{H}}\text{3}.{{b}}{}] 
	{\setlength\itemindent{10pt} \item\label{H3.b} consistency: for all $\hat q \in P_1(\mathcal{T}_{h_\circ}) \times X_{h_\partial}^{k_\partial}$
		$$\mathcal{L}_{f,\boldsymbol{h}}(\hat q) = \mathcal{L}_f(\hat q).$$} 
\end{enumerate}

Such assumption, together with some of those previously introduced, allows us to prove the following approximation error estimate for the operator $\mathcal{L}_f$.

\begin{Lemma} \label{Lemma:H3}
	Let $\hat v = (v,\mu) \in H^2(\Omega) \times H^{-\nicefrac{1}{2}}(\Gamma)$, and let 
$v_{h_\circ}^I$ be the best approximation of $v$ in  $Q_{h_\circ}^{k_\circ}$. Under assumptions \eqref{H1.a}, \eqref{H3.a} and \eqref{H3.b}, for $s \ge 1$, it holds:
	\begin{equation*}
	\abs*{\mathcal{L}_f((v_{h_\circ}^I,\mu)) - {\mathcal{L}_{f,\boldsymbol{h}}}((v_{h_\circ}^I,\mu))} \apprle h_\circ\left(h_\circ^{\text{min}(s,k_\circ)}+h_\partial^{\text{min}(s,k_\partial)}\right) \norma{v}_{H^{2}(\Omega)} \sigma_s(f).
	\end{equation*}
\end{Lemma}
\begin{proof}
	Using \eqref{H3.b} and \eqref{H3.a}, we can write
	\begin{align*}
&	\vert \mathcal{L}_f((v_{h_\circ}^I,\mu)) - {\mathcal{L}_{f,\boldsymbol{h}}}((v_{h_\circ}^I,\mu))\vert
	 \\ \hspace{0.9cm} & = 		
	\vert\mathcal{L}_f((v_{h_\circ}^I,\mu)) - 
	{\mathcal{L}_{f,\boldsymbol{h}}}((v_{h_\circ}^I,\mu)) - \mathcal{L}_f((\Pi_{1}^{\nabla}v,\mu)) + \mathcal{L}_{f,\boldsymbol{h}}((\Pi_{1}^{\nabla}v,\mu))\vert
		\\ \hspace{0.9cm} & = \vert \mathcal{L}_f((v_{h_\circ}^I - \Pi_{1}^{\nabla} v,0)) -  {\mathcal{L}_{f,\boldsymbol{h}}}((v_{h_\circ}^I - \Pi_{1}^{\nabla} v,0)) \vert\\
		\hspace{0.9cm} &\apprle \left(h_\circ^{\text{min}(s,k_\circ)}+h_\partial^{\text{min}(s,k_\partial)}\right) \| v_{h_\circ}^I - \Pi^{\nabla}_{1} v \|_{H^1(\mathcal{T}_{h_\circ})}\, \sigma_s(f).
	\end{align*}
	From \eqref{H1.a} and by using standard polynomial approximation estimates (see, for example, Lemma 4.3.8 in \cite{BrennerScott2008}), we get
	\begin{align}
	\|\Pi_{1}^{\nabla}v-v_{h_\circ}^{I}\|_{H^{1}(\mathcal{T}_{h_\circ})}\leq \|\Pi_{1}^{\nabla}v-v\|_{H^{1}(\mathcal{T}_{h_\circ})} + \|v-v_{h_\circ}^{I}\|_{H^{1}(\Omega)} \apprle h_\circ\left\|v\right\|_{H^{2}(\Omega)},\label{questa_serve}
	\end{align}
	which, combined with the previous estimate, leads to the claim.
\end{proof}
To prove the error estimate in the weaker $W$-norm, we introduce the dual space $W' := L^2(\Omega) \times H^{\nicefrac{3}{2}}(\Gamma)$ and denote by $ [\cdot,\cdot]: W \times W' \to \mathbf{R}$ the associated duality pairing. Further, we define the adjoint operator $\mathcal{A}^{*}: V\rightarrow V^{'}$ as
\begin{equation*}
\left(\mathcal{A}^{*} \hat{v}\right)(\hat{u}) := \mathcal{A}(\hat{u},\hat{v}),
\end{equation*}
which, by Lemma 3 in \cite{JohnsonNedelec1980}, is an isomorphism, whose inverse $\mathcal{A}^{*-1}: H^1(\Omega) \times H^{\nicefrac{3}{2}}(\Gamma) \rightarrow H^2(\Omega) \times H^{\nicefrac{1}{2}}(\Gamma)$ is continuous.
\begin{Theorem} 
	\label{Theorem:JN2}
	Assume that there exist $k_\circ$ and $k_\partial$ and $\sigma_s:L^2(\Omega) \to \mathbf{R}^+$ such that, for all $s \ge 1$, \eqref{H1.a},\eqref{H1.b}, \eqref{H2.a}--\eqref{H2.c}, \eqref{H3.a} and \eqref{H3.b} hold. Then, for $h_\circ$ and $h_\partial$ small enough, if $\hat u_{\boldsymbol{h}} \in V_{\boldsymbol{h}}^{\boldsymbol{k}}$ is the solution of Problem \eqref{model_problem_galerkin} and $\hat u$, solution of Problem \eqref{model_problem_operator}, satisfies $\hat u \in H^{s+1}(\Omega) \times H^{s-\nicefrac{1}{2}}(\Gamma)$, for $s \ge 1$, the following estimate
	\begin{equation*}
	\norma{\hat u - \hat u_{\boldsymbol{h}}}_W \apprle (h_\circ+h_\partial) \left(h_\circ^{\text{min}(s,k_\circ)} + h_\partial^{\text{min}(s,k_\partial)}\right) \left( \norma{u}_{H^{s+1}(\Omega)}  + \sigma_s(f)\right)
	\end{equation*}
	holds.
\end{Theorem}
 \begin{proof}
Let $\hat w \in W'$ 
and $\hat v :=\A^{*-1} \hat w$ the unique element that, according to the above mentioned property of $\A^{*-1}$, satisfies $\hat v\in H^2(\Omega) \times H^{\nicefrac{1}{2}}(\Gamma) \subset V$,
	\begin{equation} \label{eq:weaker1}
	\A(\hat z,\hat v) =\A(\A^{*-1}\hat w, \hat z) = [\hat z, \hat w] \quad \forall \hat z \in V
	\end{equation}
	and
$	\norma{\hat v}_{H^2(\Omega) \times H^{\nicefrac{1}{2}}(\Gamma)} \apprle \norma{\hat w}_{W'}$.

Now, choosing $\hat z = \hat u - \hat u_{\boldsymbol{h}}$ in \eqref{eq:weaker1}, where $\hat u$ and $\hat u_{\boldsymbol{h}}$ are the solutions of \eqref{model_problem_operator} and \eqref{model_problem_galerkin} respectively, and denoting by $\hat v_{\boldsymbol{h}}^I = (v^I_{h_\circ},\mu^I_{h_\partial}) \in V_{\boldsymbol{h}}^{\boldsymbol{k}}$ the interpolant of $\hat v$, we obtain
	\begin{align} \nonumber
	\vert\, [\hat u &- \hat u_{\boldsymbol{h}}, \hat w] \, \vert  =\abs*{\A(\hat u - \hat u_{\boldsymbol{h}}, \hat v)  } \\ \nonumber & =\abs*{\A(\hat u - \hat u_{\boldsymbol{h}}, \hat v) + \A_{\boldsymbol{h}}(\hat u_{\boldsymbol{h}}, \hat v^I_{\boldsymbol{h}}) - \A_{\boldsymbol{h}}(\hat u_{\boldsymbol{h}}, \hat v_{\boldsymbol{h}}^I) +\A(\hat u- \hat u_{\boldsymbol{h}}, \hat v_{\boldsymbol{h}}^I) -\A(\hat u- \hat u_{\boldsymbol{h}}, \hat v_{\boldsymbol{h}}^I) }\\  \nonumber
	& \leq \abs*{\A(\hat u - \hat u_{\boldsymbol{h}}, \hat v - \hat v_{\boldsymbol{h}}^I) }+\abs*{\mathcal{L}_f(\hat v_{\boldsymbol{h}}^I) - \mathcal{L}_{f,\boldsymbol{h}}(\hat v_{\boldsymbol{h}}^I)} + \abs*{\A_{\boldsymbol{h}}(\hat u_{\boldsymbol{h}}, \hat v_{\boldsymbol{h}}^I) -\A(\hat u_{\boldsymbol{h}},\hat v_{\boldsymbol{h}}^I)} \\ \nonumber
	& \apprle \norma{\hat u - \hat u_{\boldsymbol{h}}}_V \norma{\hat v - \hat v_{\boldsymbol{h}}^I}_V + h_\circ\left(h_\circ^{\text{min}(s,k_\circ)}+h_\partial^{\text{min}(s,k_\partial)}\right)  \norma{v}_{H^2(\Omega)}  \sigma_s(f) \\ & +  \abs*{\B_{\boldsymbol{h}}(\hat u_{\boldsymbol{h}}, \hat v_{\boldsymbol{h}}^I) -\B(\hat u_{\boldsymbol{h}},\hat v_{\boldsymbol{h}}^I)} =: I + II + III, \label{eq:bacca}
	\end{align}
	the last inequality following from the continuity of $\A$ and Lemma \ref{Lemma:H3}.
	By applying Theorem \ref{Theorem:JN1} and the interpolation property \eqref{int_property_0}, we estimate 
	\begin{align}\label{eq:I}
	\nonumber
	I &\apprle (h_\circ +h_\partial)\left(h_\circ^{\text{min}(s,k_\circ)} + h_\partial^{\text{min}(s,k_\partial)}\right) \left(\norma{u}_{H^{s+1}(\Omega)}  + \sigma_s(f)\right) \|\hat v\|_{H^2(\Omega) \times H^{\nicefrac{1}{2}}(\Gamma)} \\
	&\apprle (h_\circ +h_\partial)\left(h_\circ^{\text{min}(s,k_\circ)} + h_\partial^{\text{min}(s,k_\partial)}\right) \left(\norma{u}_{H^{s+1}(\Omega)}  + \sigma_s(f)\right) \|\hat w\|_{W'},
	\end{align}
	and, since $\norma{v}_{H^2(\Omega)}\le \|\hat v\|_{H^2(\Omega) \times H^{\nicefrac{1}{2}}(\Gamma)}$ we have
	\begin{equation}\label{eq:II}
	II \apprle h_\circ\left(h_\circ^{\text{min}(s,k_\circ)} + h_\partial^{\text{min}(s,k_\partial)}\right) \|\hat w\|_{W'} \sigma_s(f).
	\end{equation}
To estimate $III$ in \eqref{eq:bacca}, we add and subtract the terms $ \B_{\boldsymbol{h}}((\Pi^{\nabla}_{k_\circ}u,\lambda),\hat v_{\boldsymbol{h}}^I)$ and $\B((\Pi^{\nabla}_{k_\circ} u,\lambda),\hat v_{\boldsymbol{h}}^I)$ which, for \eqref{H2.a}, are equal.
Hence we get
	\begin{align*}
	III & = \vert
	\B_{\boldsymbol{h}}((u_{h_\circ} - \Pi^{\nabla}_{k_\circ} u,0),\hat v_{\boldsymbol{h}}^I) - \B((u_{h_\circ} - \Pi^{\nabla}_{k_\circ} u,0),\hat v_{\boldsymbol{h}}^I)\vert.
	\end{align*}
	Similarly, adding and subtracting the two equal terms (see \eqref{H3.b}) \\ $ \B_{\boldsymbol{h}}((u_{h_\circ} - \Pi_{k_\circ}^{\nabla} u,0),(\Pi^{\nabla}_{k_\circ} v, \mu_{h_\partial}^I))$ and $\B((u_{h_\circ} - \Pi_{k_\circ}^{\nabla} u,0),(\Pi^{\nabla}_{k_\circ} v, \mu_{h_\partial}^I))$, we obtain
	\begin{align*}
	III \le 
	\vert \B_{\boldsymbol{h}}((u_{h_\circ} - \Pi^{\nabla}_{k_\circ} u,0),(v_{h_\circ}^I - \Pi^{\nabla}_1 v,0))\vert + \vert \B((u_{h_\circ} - \Pi^{\nabla}_{k_\circ} u,0), (v_{h_\circ}^I - \Pi^{\nabla}_1 v,0)) \vert.
	\end{align*}
	Using the continuity of $\B_{\boldsymbol{h}}$ (see \eqref{H2.b}) and of $\B$, we have
	\begin{equation*}
	III \apprle \norma{u_{h_\circ} - \Pi^{\nabla}_{k_\circ} u}_{H^1(\mathcal{T}_{\boldsymbol{h}})} \norma{ v_{h_\circ}^I - \Pi^{\nabla}_1 v}_{H^1(\mathcal{T}_{\boldsymbol{h}})}.
	\end{equation*}
	The first factor of the above inequality is estimated, by using Theorem \ref{Theorem:JN1} and standard polynomial approximations, as follows
	\begin{align*}
	\|u_{h_\circ} - \Pi^{\nabla}_{k_\circ} u\|_{H^1(\mathcal{T}_{\boldsymbol{h}})} & \le \| u_{h_\circ} - u\|_{H^1(\Omega)} + \norma{u - \Pi^{\nabla}_{k_\circ} u}_{H^1(\mathcal{T}_{\boldsymbol{h}})}\\
	& \apprle \left(h_\circ^{\text{min}(s,k_\circ)} + h_\partial^{\text{min}(s,k_\partial)}\right) (\norma{u}_{H^{s+1}(\Omega)}  + \sigma_s(f)).
	\end{align*} 
	Then, using 
	 \eqref{questa_serve}, we obtain
	\begin{align} \label{eq:bacca1}
	\nonumber
	III &\apprle h_\circ\left(h_\circ^{\text{min}(s,k_\circ)} + h_\partial^{\text{min}(s,k_\partial)}\right) \left(\norma{u}_{H^{s+1}(\Omega)}  + \sigma_s(f)\right) \|v\|_{H^2(\Omega)}\\
	&\apprle h_\circ\left(h_\circ^{\text{min}(s,k_\circ)} + h_\partial^{\text{min}(s,k_\partial)}\right) \left(\norma{u}_{H^{s+1}(\Omega)}  + \sigma_s(f)\right) \|\hat w\|_{W'}.
	\end{align}
	Finally, the assertion easily follows combining \eqref{eq:bacca} with \eqref{eq:I}, \eqref{eq:II} and \eqref{eq:bacca1}.
\end{proof}

\section{The CVEM-BEM method}\label{sec_5_discrete_scheme}

In this section we describe the discrete CVEM-BEM coupling procedure for the solution of Problem (\ref{model_problem_operator}). 
In particular, we will show that all the assumptions, introduced in Section \ref{sec_4_galerkin_pb} and used to prove Theorems \ref{Theorem:JN1} and \ref{Theorem:JN2}, are satisfied.
Referring to the notation introduced in Section \ref{sec_4_galerkin_pb}, and denoting $P_{-1}(E) = \{0\}$, 
we consider for each $E\in\mathcal{T}_{h_\circ}$ the following local virtual space $Q^{k_\circ}_{h_\circ}(E)$ defined by
\begin{align*}
Q^{k_\circ}_{h_\circ}(E):=&\left\{ v_{h_\circ}\in H^{1}(E): \Delta v_{h_\circ}\in P_{k_\circ-2}(E),\right. \\
&\left. \hskip.4cm v_{h_\circ}\,\raisebox{-.5em}{$\vert_{e_{1}}$}\in\widetilde{P}_{k_\circ}(e_{1}), v_{h_\circ}\,\raisebox{-.5em}{$\vert_{e_{i}}$}\in P_{k_\circ}(e_{i}), \ i=2,\ldots,n_{\text{\tiny{E}}}\right\},
\end{align*}
where
$e_{1},\ldots,e_{n_{\text{\tiny{E}}}}$ denote the edges of the  boundary of $E$, whose first element $e_{1}$ is assumed to be curved and parametrized by a local map $\gamma_{E}: I_E \rightarrow e_1$,  and $\widetilde{P}_{k_\circ}(e_{1}):=\left\{\widetilde{q}=q\circ\gamma_{E}^{-1} \ : \ q\in P_{k_\circ}(I_E)\right\}$. 
\

We omit here, for brevity, the complete description of such space and we refer to \cite{BeiraoBrezziCangianiManziniMariniRusso2013, BeiraoRussoVacca2019} for a deeper presentation. Further, since we will use some of the theoretical results proved in \cite{DesiderioFallettaFerrariScuderi2021}, we also refer to this latter, in particular for what concerns the notation and the implementation details.

On the basis of the definition of the local virtual space $Q_{h_\circ}^{k_\circ}(E)$, we construct the global virtual space
$$
Q_{h_\circ}^{k_\circ} := \{ v_{h_\circ} \in H^1_{0,\Gamma_0} : v_{{h_\circ}_{\vert_E}} \in Q_{h_\circ}^{k_\circ}(E), \ E \in \mathcal{T}_{h_\circ} \},
$$
The validity of Assumption \eqref{H1.a} is based on the proof of Lemma 5.2 in \cite{DesiderioFallettaFerrariScuderi2021}, in which the results hold both for the space $Q^{k_\circ}_{h_\circ}$ and for a suitable \emph{enhanced} space associated to it.
 For each $E \in \mathcal{T}_{h_\circ}$, in the spirit of the virtual element method, we define the approximation $a^{\text{\tiny{E}}}_{h_\circ}$ of the bilinear form $a^{\text{\tiny{E}}}$ (see definition \eqref{ae}), as follows:
\begin{align*}
\label{discrete_local_bilinear_form_a} 
a^{\text{\tiny{E}}}_{h_\circ}(u_{h_\circ},v_{h_\circ})&:=a^{\text{\tiny{E}}}\left(\Pi_{k_\circ}^{\nabla,E}u_{h_\circ},\Pi_{k_\circ}^{\nabla,E}v_{h_\circ}\right)+s^{\text{\tiny{E}}}\left(\left(I-\Pi_{k_\circ}^{\nabla,E}\right)u_{h_\circ},\left(I-\Pi_{k_\circ}^{\nabla,E}\right)v_{h_\circ}\right)
\end{align*}
where $s^E$ is the standard ``\textit{dofi-dofi}'' stabilization term  (see Eq. (3.22) of \cite{BeiraoBrezziMariniRusso2014}).
\noindent
The global approximate bilinear form $a_{h_\circ}:Q^{k_\circ}_{h_\circ}\times Q^{k_\circ}_{h_\circ}\rightarrow\mathbf{R}$ is then defined by  summing up the local contributions
\begin{equation*}\label{global_discrete_bilinear_form}
a_{h_\circ}(u_{h_\circ},v_{h_\circ}):=\sum\limits_{E\in\mathcal{T}_{\boldsymbol{h}}}a^{\text{\tiny{E}}}_{h_\circ}(u_{h_\circ},v_{h_\circ}).
\end{equation*}
The boundary element space $X_{h_\partial}^{k_\partial}$ associated to the artificial boundary $\Gamma$ is defined as follows:
\begin{equation*}
X_{h_\partial}^{k_\partial}:=\left\{\lambda \in L^{2}(\Gamma) \ : \lambda_{\mkern 1mu \vrule height 2ex\mkern2mu e}\in\widetilde{P}_{k_\partial-1}(e), \  e\in\mathcal{T}_{h_\partial}^{\Gamma}\right\} \cap \,  H_0^{-\nicefrac{1}{2}}(\Gamma),
\end{equation*}
and we refer to \cite{LeRoux1977} for the validity of the associated Assumption \eqref{H1.b}.
We then define the approximate bilinear form $\mathcal{B}_{\boldsymbol{h}}:V^{\boldsymbol{k}}_{\boldsymbol{h}}\times V^{\boldsymbol{k}}_{\boldsymbol{h}}\rightarrow\mathbf{R}$ as:
\begin{equation*} \label{eq:modifiedB}
\mathcal{B}_{\boldsymbol{h}}(\hat{u}_{\boldsymbol{h}},\hat{v}_{\boldsymbol{h}}):= a_{h_\circ}(u_{h_\circ},v_{h_\circ}) - \langle\lambda_{h_\partial},v_{h_\circ}\rangle_{\Gamma} + \langle \mu_{h_\partial},u_{h_\circ}\rangle_{\Gamma} + 2\langle\mu_{h_\partial},\text{V} \lambda_{h_\partial}\rangle_{\Gamma}
\end{equation*}
for $\hat u_{\boldsymbol{h}} = (u_{h_\circ}, \lambda_{h_\partial}), \hat v_{\boldsymbol{h}} = (v_{h_\circ},\mu_{h_\partial})\in V^{\boldsymbol{k}}_{\boldsymbol{h}}$.

For these choices, from \cite{DesiderioFallettaFerrariScuderi2021} (see Section 5.2), Assumptions \eqref{H2.a}-\eqref{H2.c} are satisfied .

By approximating the linear operator $\mathcal{L}_f$ in a standard VEM way, in particular as in \cite{BrennerGuanSung2017} (see Eq. (3.18)), and assuming $f \in H^{s-1}(\Omega)$, from Lemma 3.4 in \cite{BrennerGuanSung2017}, Assumption (\ref{H3.a}) follows with $\sigma_s(f) = \vert f\vert_{H^{s-1}(\Omega)}$. Finally, Assumption \eqref{H3.b} is trivially satisfied. 

\section{Algebraic details and computational issues}\label{sec_6_impl_det}

In this section we briefly detail the construction of the final linear system associated with the CVEM-BEM scheme, and we give some implementation issues concerning the BEM matrices. 

We start by re-ordering and splitting the complete index set $\Stot$  of the basis functions $\left\{\Phi_{j}\right\}_{j\in\Stot}$ of $Q^{k_\circ}_{h_\circ}$ as
$\Stot=\Sint\cup\SB$,
where $\Sint$ and $\SB$ denote the sets of indices related to the internal degrees of freedom and to those lying on $\Gamma$, respectively. Moreover, we denote by  $\left\{\phi_{j}\right\}_{j\in\mathcal{G}}$ the basis functions of $X^{k_\partial}_{h_\partial}$, $\mathcal{G}$ being the corresponding index set.
In order to write the linear system associated with the discrete problem \eqref{model_problem_galerkin}, we expand the unknown function $\hat u_{\boldsymbol{h}} = (u_{h_\circ}, \lambda_{h_\partial}) \in Q_{h_\circ}^{k_\circ} \times X^{k_\partial}_{h_\partial}$  as
\begin{equation}\label{interpolants}
\begin{aligned}
&u_{h_\circ}(\mathbf{x})=:\sum\limits_{j\in\Stot}u_{h_\circ}^{j}\Phi_{j}(\mathbf{x}) \quad \text{with} \quad u_{h_\circ}^{j}=\text{dof}_{j}(u_{h_\circ}),\\
&\lambda_{h_\partial}(\mathbf{x})=:\sum\limits_{j\in\mathcal{G}}\lambda_{h_\partial}^{j}\varphi_j(\mathbf{x}) \quad \text{with} \quad \lambda_{h_\partial}^{j}=\text{dof}_{j}(\lambda_{h_\partial}).
\end{aligned}
\end{equation}
Hence, using the basis functions of $Q^{k_\circ}_{h_\circ}$ to test the discrete counterpart of equation \eqref{model_problem_variational_1}, we get for $i\in\Stot$
\begin{align}\label{VEM_discrete_equation}
\sum\limits_{j\in\Stot}u_{h_\circ}^{j}&\sum\limits_{E\in\mathcal{T}_{h_\circ}}a^{\text{\tiny{E}}}_{h_\circ}(\Phi_{j},\Phi_{i})-\sum\limits_{j\in\mathcal{G}}\lambda_{h_\partial}^{j}\langle\varphi_{j},\Phi_{i}\rangle_{\Gamma} = \mathcal{L}_{f,\boldsymbol{h}}((\Phi_i,0)).
\end{align}
To write the matrix form of the above linear system, we introduce the stiffness matrix $\mathbb{A}$ and the matrix $\mathbb{Q}$ whose entries are respectively defined by
\begin{displaymath}
\mathbb{A}_{ij} :=  \sum\limits_{E\in\mathcal{T}_{h_\circ}}a^{\text{\tiny{E}}}_{h_\circ}(\Phi_{j},\Phi_{i}), \qquad 
\mathbb{Q}_{ij} :=  \langle\varphi_{j},\Phi_{i}\rangle_{\Gamma}
\end{displaymath}
and the column vectors $\mathbf{u} = \left[u_{h_\circ}^{j}\right]_{j\in\Stot}$, $\boldsymbol{\lambda} = \left[\lambda_{h_\partial}^{j}\right]_{j\in\mathcal{G}}$ and $\mathbf{f} = \left[\mathcal{L}_{f,\boldsymbol{h}}((\Phi_i,0))\right]_{i\in\Stot}$.
In accordance with the splitting of the set of the degrees of freedom, we consider the block partitioned representation of the above matrices and vectors (with obvious meaning of the notation), and we rewrite equations \eqref{VEM_discrete_equation} as follows:
\begin{eqnarray}\label{VEM_system}
\left[
\begin{array}{ll}
\mathbb{A}^{I I}& \mathbb{A}^{I \Gamma}\\
&\\
\mathbb{A}^{\Gamma I} & \mathbb{A}^{\Gamma \Gamma} \\
\end{array}\right]\left[\begin{array}{l}
\mathbf{u}^I\\
\\
\mathbf{u}^\Gamma
\end{array}
\right]
-	\left[
\begin{array}{l}
\mathbf{0} \\
\\
\mathbb{Q} \boldsymbol{\lambda}
\end{array}
\right]
= 	\left[
\begin{array}{l}
\mathbf{f}^{I}\\
\\
\mathbf{f}^\Gamma
\end{array}\right].
\end{eqnarray}
For what concerns the discretization of the BI-NRBC, by inserting \eqref{interpolants} in \eqref{model_problem_variational_2} and testing with the functions $\varphi_i$, $i \in\mathcal{G}$, we obtain
\begin{equation}\label{boundary_integral_equation_discretized}
\begin{aligned}
\sum\limits_{j\in\SB}&u_{h_\circ}^{j}\left[\frac{1}{2}\int\limits_{\Gamma}\Phi_{j}(\mathbf{x})\varphi_{i}(\mathbf{x})\dd\Gamma_\mathbf{x}-\int\limits_{\Gamma}\left(\int\limits_{\Gamma}\frac{\partial G}{\partial\mathbf{n}_{\mathbf{y}}}(\mathbf{x},\mathbf{y})\Phi_{j}(\mathbf{y})\,\dd\Gamma_{\mathbf{y}}\right)\varphi_{i}(\mathbf{x})\dd\Gamma_{\mathbf{x}}\right]\\
&\,\,+\sum\limits_{j\in\mathcal{G}}\lambda_{h_\partial}^{j}\int\limits_{\Gamma}\left(\int\limits_{\Gamma}G(\mathbf{x},\mathbf{y})\varphi_{j}(\mathbf{y})\,\dd\Gamma_{\mathbf{y}}\right)\varphi_{i}(\mathbf{x})\dd\Gamma_{\mathbf{x}}=0, 
\end{aligned}
\end{equation}
that in matrix form reads
\begin{equation}\label{BEM_system}
\left(\frac{1}{2}\mathbb{Q}^T-\mathbb{K}\right)\mathbf{u}^{\Gamma}+\mathbb{V}\boldsymbol{\lambda}=\mathbf{0},
\end{equation}
where
\begin{align*}
\mathbb{V}_{ij}&:=\int\limits_{\Gamma}\left(\int\limits_{\Gamma}G(\mathbf{x},\mathbf{y})\phi_{j}(\mathbf{y})\dd\Gamma_{\mathbf{y}} \right) \phi_{i}(\mathbf{x})\dd\Gamma_{\mathbf{x}},\\
\mathbb{K}_{ij}&:=\int\limits_{\Gamma}\left(\int\limits_{\Gamma}\frac{\partial G}{\partial\mathbf{n}_{\mathbf{y}}}(\mathbf{x},\mathbf{y})\Phi_{j}(\mathbf{y})\dd\Gamma_{\mathbf{y}}\right)\phi_{i}(\mathbf{x})\dd\Gamma_{\mathbf{x}}.
\end{align*}
By combining \eqref{VEM_system} with \eqref{BEM_system} we obtain the final linear system
\begin{eqnarray}\label{final_system}
\left[
\begin{array}{ccc}
\mathbb{A}^{I I}& \mathbb{A}^{I \Gamma} & \mathbb{O}\\

&\\
\mathbb{A}^{\Gamma I} & \mathbb{A}^{\Gamma \Gamma} & -\mathbb{Q}\\
&\\
\mathbb{O}   & \frac{1}{2}\mathbb{Q}^T-\mathbb{K}& \mathbb{V}
\end{array}\right]\left[\begin{array}{l}
\mathbf{u}^I\\
\\
\mathbf{u}^\Gamma\\
\\
\boldsymbol{\lambda}
\end{array}
\right]
= 	\left[
\begin{array}{c}
\mathbf{f}^{I}\\
\\
\mathbf{f}^\Gamma\\
\\
\mathbf{0}
\end{array}\right].
\end{eqnarray}
It is worth to point out that, since in the theoretical analysis we have assumed that the boundary integral operators are not approximated, it is crucial to compute the integrals defining the BEM entries of $\mathbb{V}$ and $\mathbb{K}$ with a high accuracy. Hence, to retrieve their approximation without affecting the overall accuracy of the coupled CVEM-BEM scheme, suitable efficient quadrature formulas must be considered. In \cite{DesiderioFallettaScuderi2021} we have proposed and successfully applied a smoothing technique to weaken the $\log$- singularity of the single layer operator and to compute the corresponding entries with high accuracy by using the Gauss-Legendre product quadrature rule with few nodes. Such a strategy has been tuned for the standard nodal linear and quadratic basis functions and could be, in principle, adopted also in this context for the basis functions satisfying property \eqref{eq:identity}. However, since the strategy associated to the standard Lagrangian basis is a well-established task, we take advantage of it by applying a computational trick. To describe this latter, we introduce the space 
$$
\widehat X_{h_\partial}^{k_\partial} := \{ \lambda \in L^2(\Gamma) : {\lambda}_{\vert_e} \in \widetilde P_{k_\partial-1}(e), \, \, e \in \mathcal{T}_{h_\partial}^\Gamma \} \subset H^{-\nicefrac{1}{2}}(\Gamma),
$$
whose Lagrangian basis functions are denoted by $\widehat \phi_{j}$.

Further, we denote by $\widehat{\mathbb{V}}, \widehat{\mathbb{K}}$ and $\widehat{\mathbb{Q}}$ the matrices associated with the choice of the space $\widehat X_{h_\partial}^{k_\partial}$, which differ from ${\mathbb{V}}, {\mathbb{K}}$ and ${\mathbb{Q}}$ by the presence of the functions $\widehat \phi_i$ instead of $\phi_i$. 
In the forthcoming Remark \ref{rk:detail} we detail the quadrature adopted to efficiently compute $\widehat{\mathbb{V}}$.
Here we describe how to retrieve the matrices $\mathbb{V}, \mathbb{K}$ and $\mathbb{Q}$ from the corresponding $\widehat{\mathbb{V}}, \widehat{\mathbb{K}}$ and $\widehat{\mathbb{Q}}$. To this aim it is sufficient to define the functions $\phi_{i}$ as a suitable linear combination of the standard ones $\widehat\phi_{i}$ and, hence, to combine accordingly the rows and/or the columns of $\widehat{\mathbb{V}}, \widehat{\mathbb{K}}$ and $ \widehat{\mathbb{Q}}$.
Such a combination depends on the order $k_\partial$, on the shape of the artificial boundary  $\Gamma$ and on the associated mesh $\mathcal{T}_{h_\partial}^{\Gamma}$. In particular, for $k_\partial = 2$, we define $\phi_{i}:=  c_i\widehat\phi_{i} +\widehat \phi_{i+1}$, where $\widehat \phi_{i}$ and $\widehat \phi_{i+1}$ are two consecutive piece-wise linear nodal basis functions.
For $k_\partial = 3$, we distinguish the following two cases: a) $\phi_{2i}:=  c_{2i}\widehat\phi_{2i} + \widehat \phi_{2i+1}$; b)  $\phi_{2i+1}:=  \widehat\phi_{2i+1} + c_{2i+1}\widehat \phi_{2i+2}$.
 The coefficients $c_{\ell}$ are chosen such that the relation $\int_\Gamma \phi_{\ell} = 0$ is satisfied and are retrieved by applying a $\nu$-point Gauss-Legendre quadrature formula, with $\nu$ chosen such that the integral over $\Gamma$ of the associated $\widehat \varphi$ functions is accurately computed. It is worth noting that $\text{dim}(X_{h_\partial}^{k_\partial}) = \text{dim}(\widehat X_{h_\partial}^{k_\partial}) -1$. 

In Figures \ref{fig:phi1} and \ref{fig:phi2}, we show the basis functions of the spaces $\widehat X_{h_\partial}^{k_\partial}$ and $X_{h_\partial}^{k_\partial}$, with ${k_\partial} = 2,3$ respectively, associated with the  uniformly partitioned parametrization interval $[0,2\pi)$ of the particular choice of a circumference. For this choice, it is immediate to get $c_i = -1$ for ${k_\partial} = 2$,  and  $c_{2i} = c_{2i+1} = -2$ for ${k_\partial} = 3$.
\begin{figure}[h]
  \centering
  \begin{minipage}[b]{0.44\textwidth}
    \includegraphics[width=\textwidth]{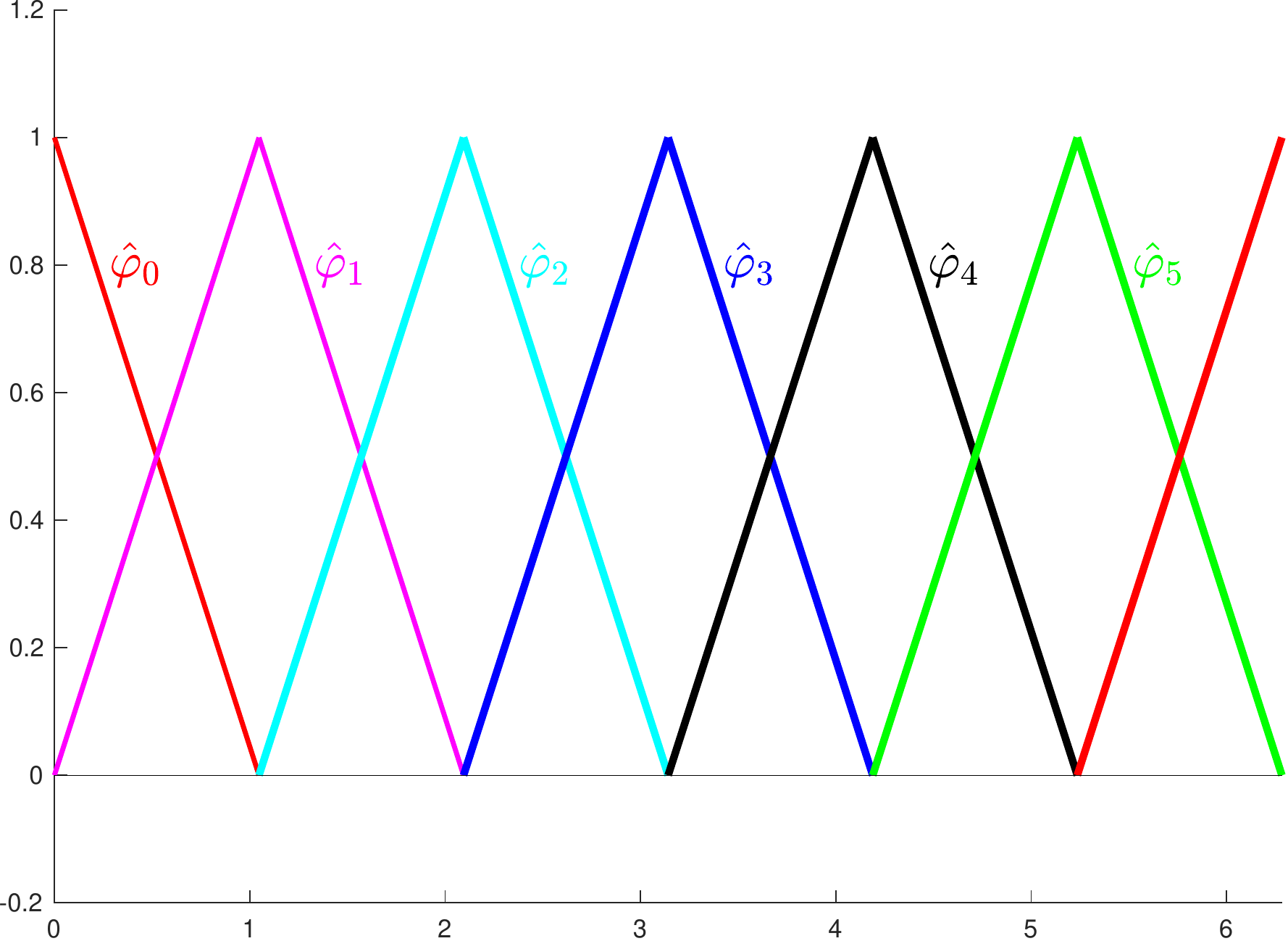}    	
  \end{minipage}
  \hfill
  \begin{minipage}[b]{0.44\textwidth}
    \includegraphics[width=\textwidth]{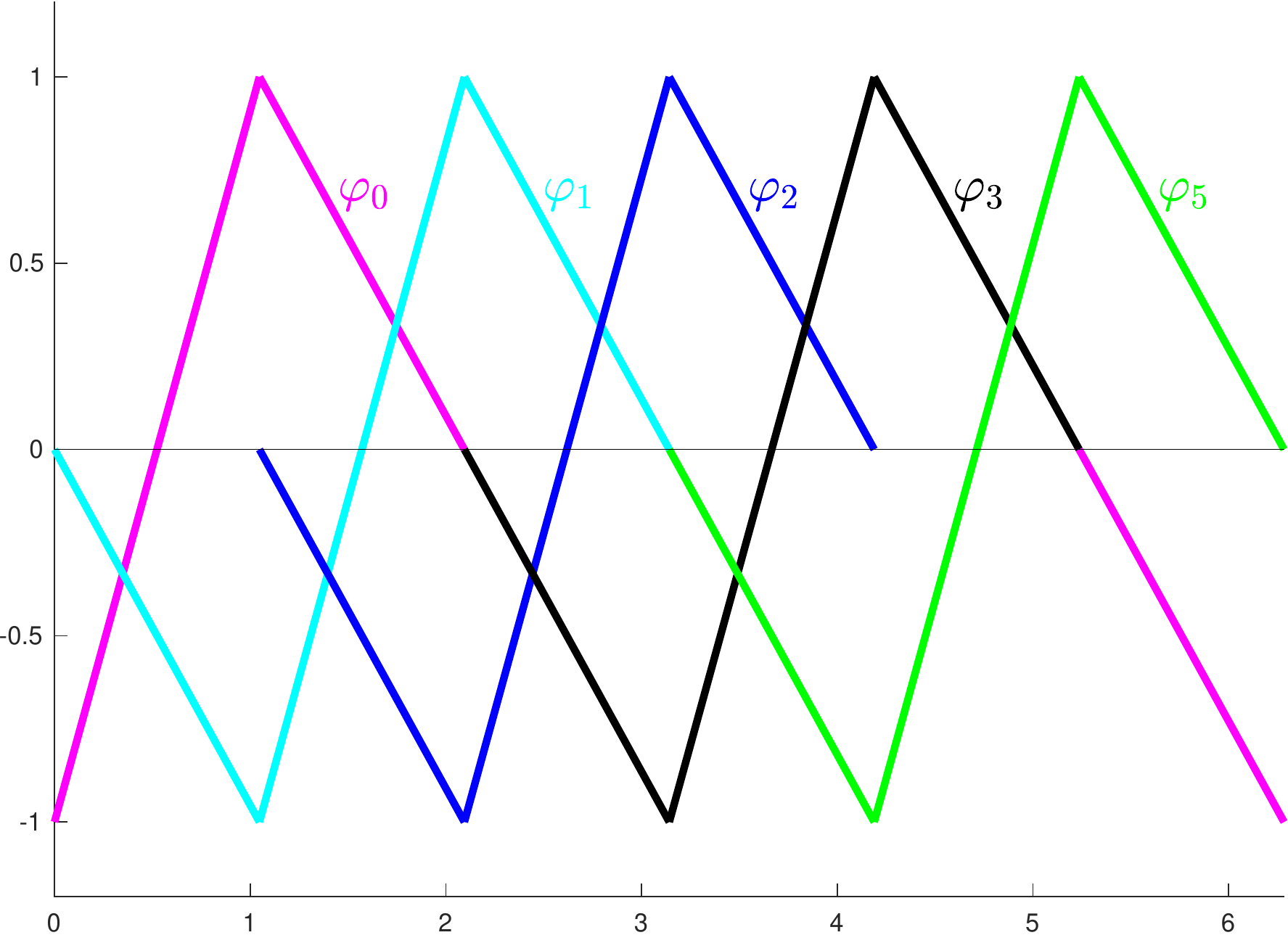}
  \end{minipage}
  	\caption{Basis of $\hat X_{h_\partial}^2$ and $X_{h_\partial}^2$ in $[0,2\pi)$}\label{fig:phi1}
\end{figure}
\begin{figure}[h]
  \centering
  \begin{minipage}[b]{0.48\textwidth}
    \includegraphics[width=\textwidth]{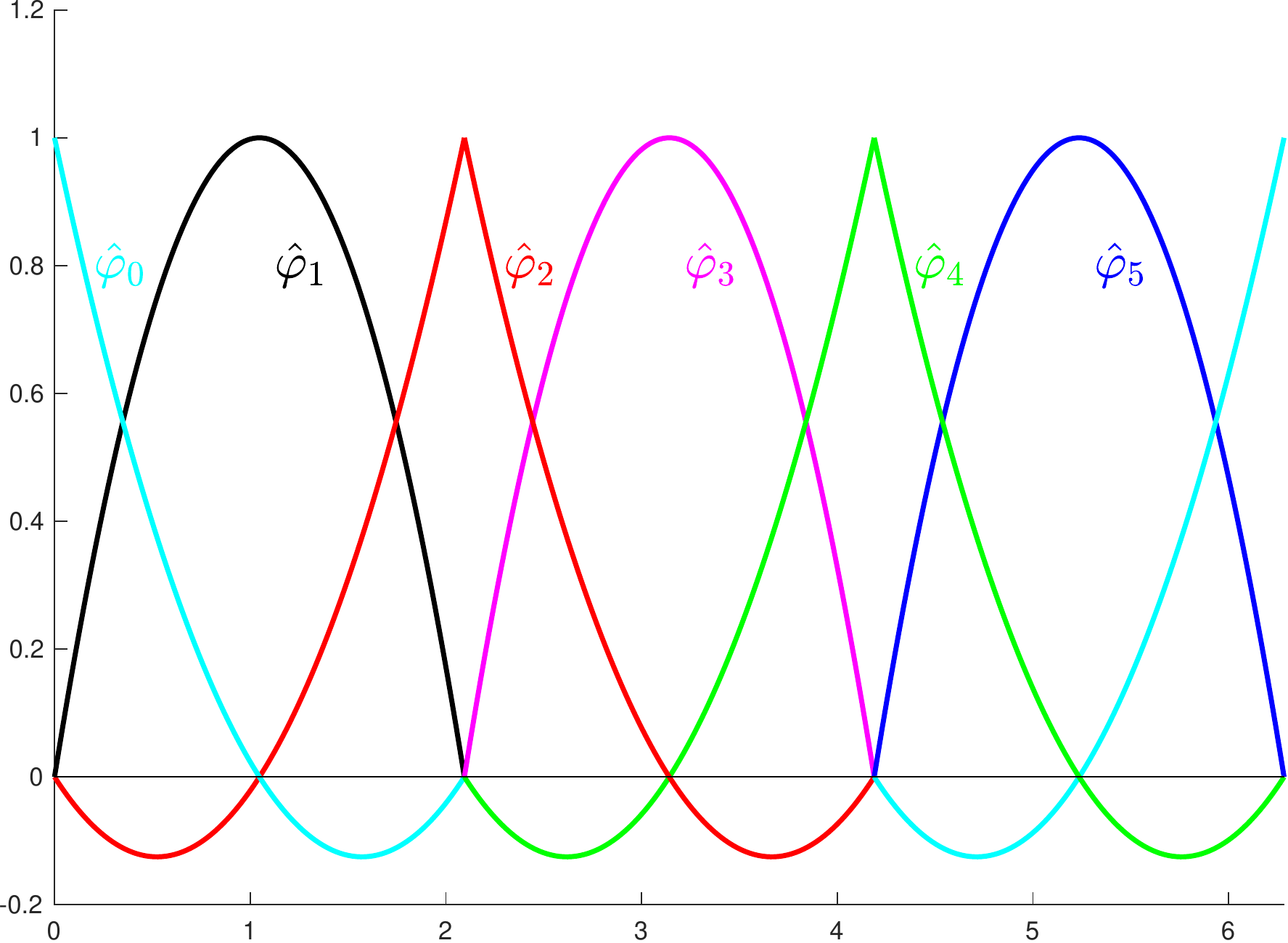}    	
  \end{minipage}
  \hfill
  \begin{minipage}[b]{0.48\textwidth}
    \includegraphics[width=\textwidth]{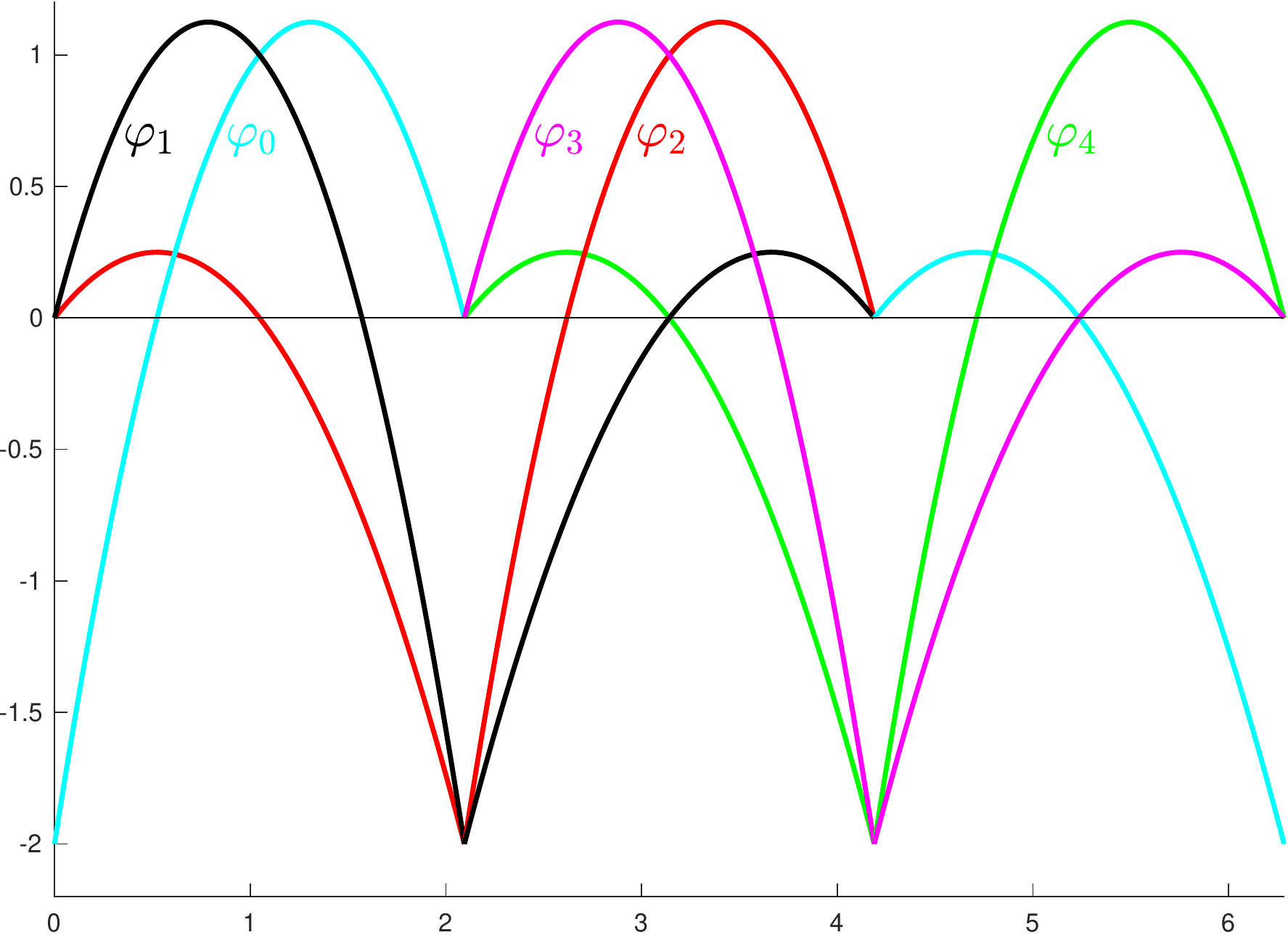}
  \end{minipage}
  	\caption{$\hat X_{h_\partial}^3$ and $X_{h_\partial}^3$ in $[0,2\pi)$}\label{fig:phi2}
\end{figure}
%
%%%%%%%%%%%%%%
\begin{remark}\label{rk:detail}
 We recall that the numerical integration difficulties in the computation of the $\widehat{\mathbb{V}}$ entries spring from the $\log$-singularity of $G(r)$ near the origin, the latter being the kernel of the single layer operator $\text{V}$. To compute the corresponding integrals with high accuracy by few nodes, we have used the very simple and efficient polynomial smoothing technique
proposed in \cite{MonegatoScuderi1999}, referred as the \emph{q-smoothing} technique. It is worth noting that such technique is applied only when the distance $r$ approaches zero. This case  corresponds to the matrix entries belonging to the main diagonal and to the co-diagonals for which the supports of the basis functions overlap or are contiguous. After having introduced the \emph{q-smoothing} transformation, with $q=3$, we have applied the $n$-point Gauss-Legendre quadrature rule with  $n=9$ for the outer integrals, and $n=8$ for the inner ones (see \cite{FallettaMonegatoScuderi2014} and Remark 3 in  \cite{DesiderioFallettaScuderi2021} for further details). For the computation of all the other integrals, we have applied a $9\times 8$-point Gauss-Legendre product quadrature rule. Incidentally, we point out that the integrals involving the kernel function $\partial_{\bf n}{G}$, appearing in the double layer operator $\text{K}$, do not require a particular quadrature strategy, since its singularity $1/r$ is factored out by the same behaviour of $\partial_\mathbf{n} r$ near the origin. Hence, for the computation of the entries of the matrix $\widehat{\mathbb{K}}$, we have directly applied a $9\times 8$-point Gauss-Legendre product quadrature rule.
The quadrature strategy described above guarantees the computation of all the mentioned integrals with a full precision accuracy (16-digit double precision arithmetic) for both $k_\partial=2$ and $k_\partial=3$.
\end{remark}

\section{Numerical results}\label{sec_7_num_results}
In this section, we present some numerical test to validate the theoretical results and to show the effectiveness of the proposed method.

For the generation of the partitioning $\mathcal{T}_{h_\circ}$ of the computational domain $\Omega$, we have used the GMSH software to construct unstructured conforming meshes consisting of quadrilaterals (see \cite{GeuzaineRemacle2009}). If a polygon $E$ has a (straight) edge bordering with the interior boundary or with the artificial one, we transform it into a curved boundary edge by means of a suitable parametrization. We remark that, even if in principle it is possible to fully decouple the interior and boundary meshes, we consider here for simplicity the boundary mesh inherited by the interior one, for which it turns out $h_\partial \leq h_\circ$. 
Furthermore, we point out that in all the numerical test we have considered $k_\partial = 2,3$; larger values than those considered would require a tailored quadrature technique for the accurate computation of the BEM matrices that we have not performed yet. Since this usually is considered the bottleneck of the BEM, the use of decoupled approximation orders allows us to exploit the flexibility of the CVEM to retrieve high accuracy by increasing only the approximation order $k_\circ$. In Example \ref{ex1} we will investigate this aspect.

\subsection{Example 1}\label{ex1}
Let us consider the unbounded region $\Omega_{\text{e}}$, external to the unitary disk
$\Omega_{0}=\{\mathbf{x}=(x_{1},x_{2})^{\top}\in\mathbf{R}^{2} \ : \ x_{1}^{2}+x_{2}^{2}\leq1\}$.
We consider Problem (\ref{dirichlet_problem}) with $f = 0$ and
$g(\mathbf{x}) =  x_1 + 2$
prescribed on the boundary $\Gamma_{0} = \partial \Omega_0$. In this case, the exact solution $u(\mathbf{x})$ is known and its expression is given by 
\begin{equation*}\label{example_sol}
u(\mathbf{x})=\frac{x_1}{x_1^2+x_2^2} + 2, \quad \ \mathbf{x} \in\Omega_{\text{e}}.
\end{equation*}
We choose as artificial boundary $\Gamma$ the circumference of radius 2, 
 so that the finite computational domain $\Omega$ is the annulus bounded internally by $\Gamma_{0}$ and externally by $\Gamma$.  

To develop a convergence analysis, we start by considering a coarse mesh, associated to the level of refinement zero (lev. 0), and all the successive refinements are obtained by halving each side of its elements.
In Table \ref{tab:prima}, we report the total number of the degrees of freedom associated to the CVEM space, corresponding to each decomposition level of the computational domain, and the approximation orders $k_\circ= k_\partial = 2,3$ (see Figure \ref{fig:circ_circ_meshes} for the meshes corresponding to level 0 (left plot) and level 2 (right plot)).% 
\begin{figure}[h!]
	\centering
	\includegraphics[width=0.50\textwidth]{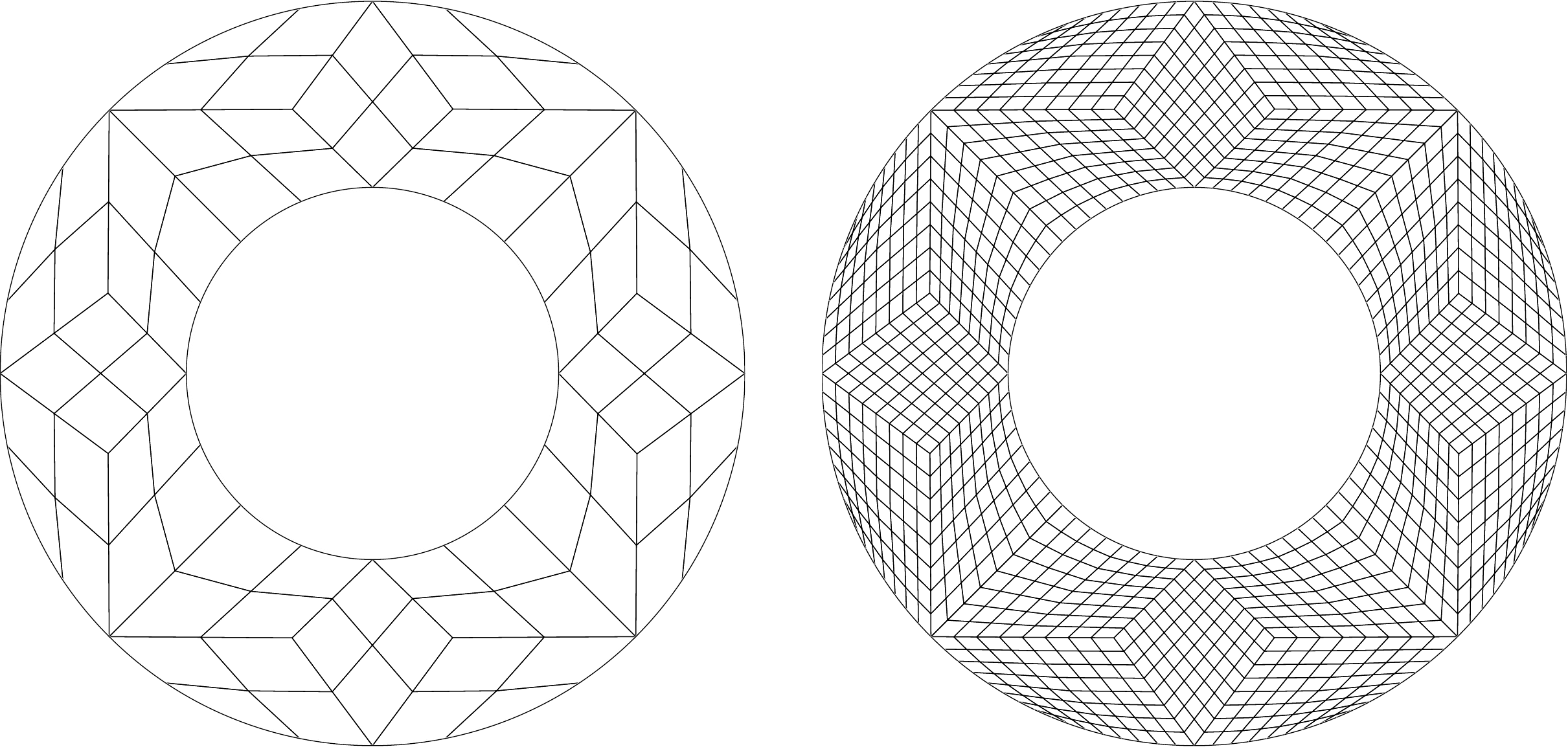}
	\caption{Meshes of $\Omega$ for lev. 0 (left plot) and lev. 2 (right plot).}
	\label{fig:circ_circ_meshes}
\end{figure}

\begin{table}[h!]
	\centering
	\begin{tabular}{lcccccc}
		\toprule
		& lev. 0 & lev. 1 & lev. 2 & lev. 3	& lev. 4 & lev. 5 \\ 
		\toprule
		$k_\circ=2$ & $368$	& $1,376$  & $5,312$ & $20,864$ & 	$82,688$ & $329,216$ \\
		$k_\circ=3$ & $792$	& $3,024$  & $11,808$ & $46,656$ & $185,472$ & $739,584$  \\
		\bottomrule
	\end{tabular}
	\caption{Number of the degrees of freedom associated to the CVEM space.}
	\label{tab:prima}
\end{table}
\indent
To test our numerical approach and to validate the theoretical analysis, the order $k_\circ$ of the approximation spaces is chosen equal to 2 (quadratic) and 3 (cubic), and $k_\partial = k_\circ$. Moreover, recalling that the approximate solution $u_{h_\circ}$ is not known inside the polygons, as suggested in \cite{BeiraoRussoVacca2019} we compute the $H^1$-seminorm and $L^2$-norm relative errors, and the corresponding EOC, by means of the following formulas:
\vspace{0.2cm}
\begin{itemize}
	\item $H^{1}$-seminorm \hspace{0.35cm}  ${\small \varepsilon^{\nabla,k_\circ}_{\text{lev}}:=\sqrt{\frac{\sum\limits_{E\in\mathcal{T}_{h_\circ}}\abs*{u-\Pi_{k_\circ}^{\nabla,E}u_{h_\circ}}^{2}_{H^{1}(E)}}{\abs*{u}^{2}_{H^{1}(\Omega)}}}}$; \vspace{0.2cm}
	\item $L^{2}$-norm \hspace{1.15cm} $\varepsilon^{0,k_\circ}_{\text{lev}}:=\sqrt{\frac{\sum\limits_{E\in\mathcal{T}_{h_\circ}}\left\|u-\Pi_{k_\circ}^{0,E}u_{h_\circ}\right\|^{2}_{L^{2}(E)}}{\left\|u\right\|^{2}_{L^{2}(\Omega)}}}$;
	\vspace{0.1cm}
	\item $\text{EOC}:=\log_{2}\left(\frac{\varepsilon^{*,k_\circ}_{\text{lev}+1}}{\varepsilon^{*,k_\circ}_{\text{lev}}}\right), \quad *  \in \{\nabla,0\}$.
\end{itemize}
In the above formulas the superscript $k_\circ=2,3$ refers to the approximation order of $u$, and the subscript $\text{lev}$ refers to the refinement level. 
For what concerns the evaluation of these errors, to compute the associated integrals over polygons we have used the $n$-point quadrature formulas proposed in \cite{SommarivaVianello2007} and \cite{SommarivaVianello2009}, which are exact for polynomials of degree at most $2n$. For curved polygons, we have applied the generalization of these formulas suggested in \cite{BeiraoRussoVacca2019} (see Section 4.3). In both cases, we have chosen $n=8$.

In Table \ref{tab:convergence_circ_circ_CVEM_BEM_kappa_1} we report $\varepsilon^{\nabla,k_\circ}_{\text{lev}}$ and  $\varepsilon^{0,k_\circ}_{\text{lev}}$ and the corresponding EOC by varying the refinement level from 0 to 5. As we can see the $H^{1}$-seminorm error and the $L^{2}$-norm error estimates confirm the expected convergence order of the method. 
\begin{table}[h!]
	\centering
	\def\sym#1{\ifmmode^{#1}\else\(^{#1}\)\fi}
	\scalebox{0.95}{
		\begin{tabular}{lc||cccc||cccc}
			\toprule
			\multicolumn{2}{c}{} & \multicolumn{4}{c}{$L^{2}$-norm} & \multicolumn{4}{c}{$H^{1}$-seminorm}\\
			\toprule
			lev.	& $h_\circ$ & $\varepsilon^{0,2}_{\text{lev}}$ & $\text{EOC}$ & $\varepsilon^{0,3}_{\text{lev}}$	& $\text{EOC}$ & $\varepsilon^{\nabla,2}_{\text{lev}}$ & $\text{EOC}$ & $\varepsilon^{\nabla,3}_{\text{lev}}$	& $\text{EOC}$\\ 	
			\toprule
			$0$ & $8.02e-01$ & $4.26e-04$ & $$ & $6.74e-05$ & $$ & $4.96e-04$ & $$ & $1.05e-04$ & $$\\
			$$  &  $$  &  $$  &  $2.9$  &  $$  &  $3.9$  &  $$  &  $1.9$  &  $$  &  $2.8$\\
			$1$ & $4.28e-01$ & $5.56e-05$ & $$ & $4.58e-06$ & $$ & $1.36e-04$ & $$ & $1.51e-05$ & $$\\
			$$  &  $$  &  $$  &  $3.0$  &  $$  &  $4.0$  &  $$  &  $2.0$  &  $$  &  $3.0$\\
			$2$ & $2.22e-01$ & $7.05e-06$ & $$ & $2.92e-07$ & $$ & $3.46e-05$ & $$ & $1.95e-06$ & $$\\
			$$  &  $$  &  $$  &  $3.0$  &  $$  &  $4.0$  &  $$  &  $2.0$  &  $$  &  $3.0$\\
			$3$ & $1.13e-01$ & $8.82e-07$ & $$ & $1.84e-08$ & $$ & $8.68e-06$ & $$ & $2.45e-07$ & $$\\
			$$  &  $$  &  $$  &  $3.0$  &  $$  &  $4.0$  &  $$  &  $2.0$  &  $$  &  $3.0$\\
			$4$ & $5.68e-02$ & $1.10e-07$ & $$ & $1.14e-09$ & $$ & $2.17e-06$ & $$ & $3.07e-08$ & $$\\
			$$  &  $$  &  $$  &  $3.0$  &  $$  &  $4.0$  &  $$  &  $2.0$  &  $$  &  $3.0$\\
			$5$ & $2.85e-02$ & $1.38e-08$ & $$ & $7.35e-11$ & $$ & $1.35e-07$ & $$ & $3.93e-09$ & $$\\
			\bottomrule
		\end{tabular}
	}
	\caption{Example 1. Relative errors and EOC.}
	\label{tab:convergence_circ_circ_CVEM_BEM_kappa_1}
\end{table}
For this example, we further investigate the possibility of choosing different approximation orders. In particular, since the meshes we have considered to generate Table \ref{tab:convergence_circ_circ_CVEM_BEM_kappa_1} have the property $h_\circ = 2 h_\partial$, we analyse the convergence of the scheme by fixing $k_\partial$ and varying $k_\circ$.
In Figures  \ref{fig:H1} and \ref{fig:L2} we report the behaviour of the $H^1$-seminorm and $L^2$-norm relative error, respectively.
For each of them we fix in the left plots $k_\partial = 2$ and in the right ones $k_\partial = 3$ and we report the errors associated to the refinement levels 0,1,2 by varying $k_\circ$. As we can see the CVEM convergence order dominates on the BEM one for each $k_\circ \leq 4$, while for larger values the BEM error is no longer negligible. Further, we observe that for $k_\partial = 3$ and $k_\circ = 5$ the CVEM $H^1$-convergence order is preserved, contrarily to that of the $L^2$ one.
It is worth to point out that, for fixed values of $h_\circ$ and $h_\partial$ and for fixed $k_\partial$, the maximum value of $k_\circ$ such that the CVEM error is larger than that of the BEM is related to the dependence of the implicit constants of the error estimates on $k_\circ$ and $k_\partial$. We are aware of a study on such dependency for some interior VEM problems (see for example \cite{BeiraoChernovMascottoRusso2016}, \cite{BeiraoChernovMascottoRusso2018} and \cite{BeiraoManziniMascotto2019}). This is a task by no mean trivial and is worth to be investigated.
Finally, as we can see from Table \ref{tab:prima}, while the increasing behaviour of $k_\circ$ for a fixed mesh is approximately linear, that of lev. for a fixed $k_\circ$ is quadratic. Therefore, it is worth noting that it is more efficient, in terms of computational cost and memory saving, to use high order CVEM rather than to refine the mesh, the latter choice being also computationally demanding for what concerns the efficient computation of the BEM matrices.
\begin{figure}[h]
  \centering
  \begin{minipage}[b]{0.42\textwidth}
    \includegraphics[width=\textwidth]{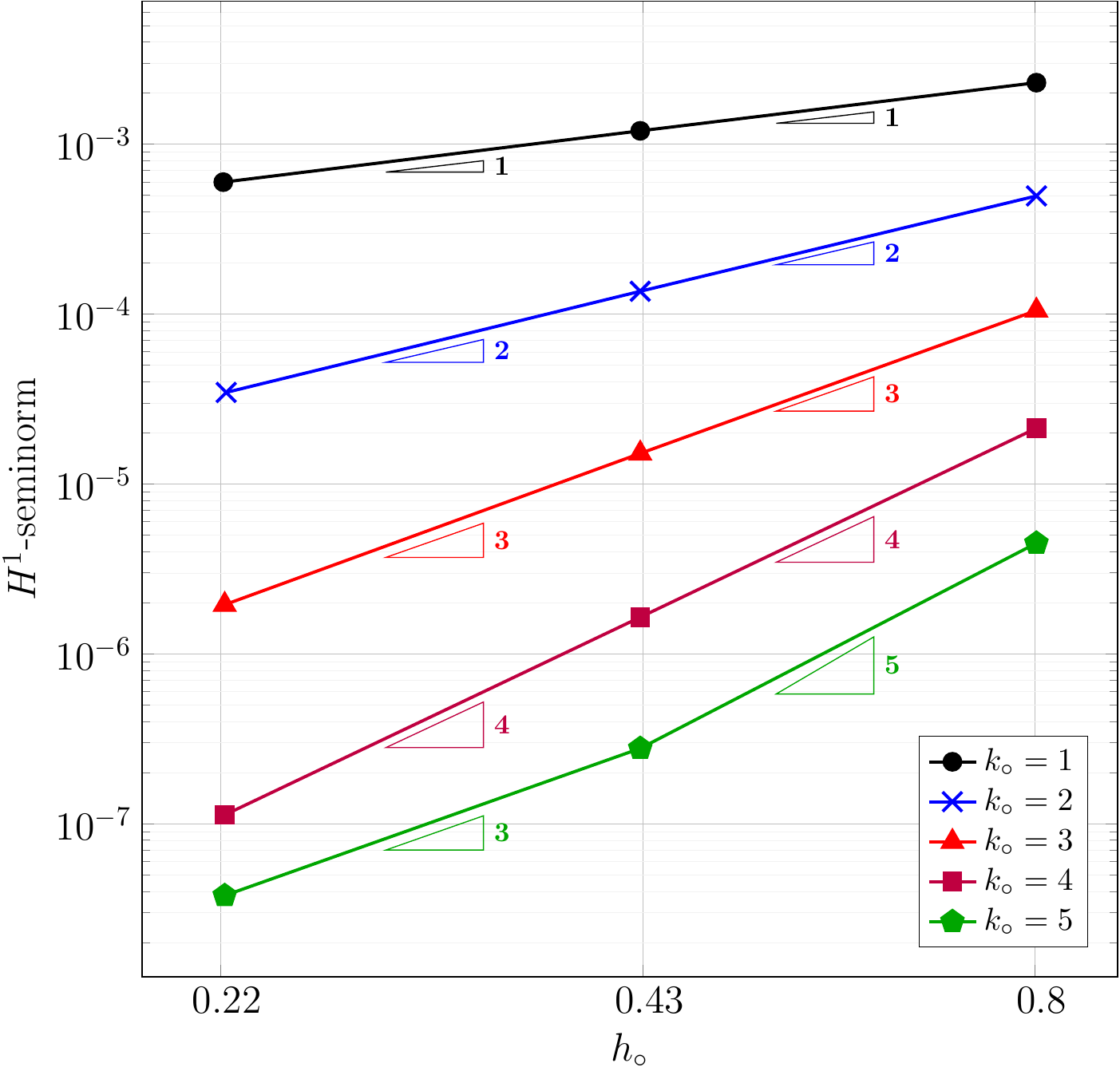}
  \end{minipage}
  \hfill
  \begin{minipage}[b]{0.42\textwidth}
    \includegraphics[width=\textwidth]{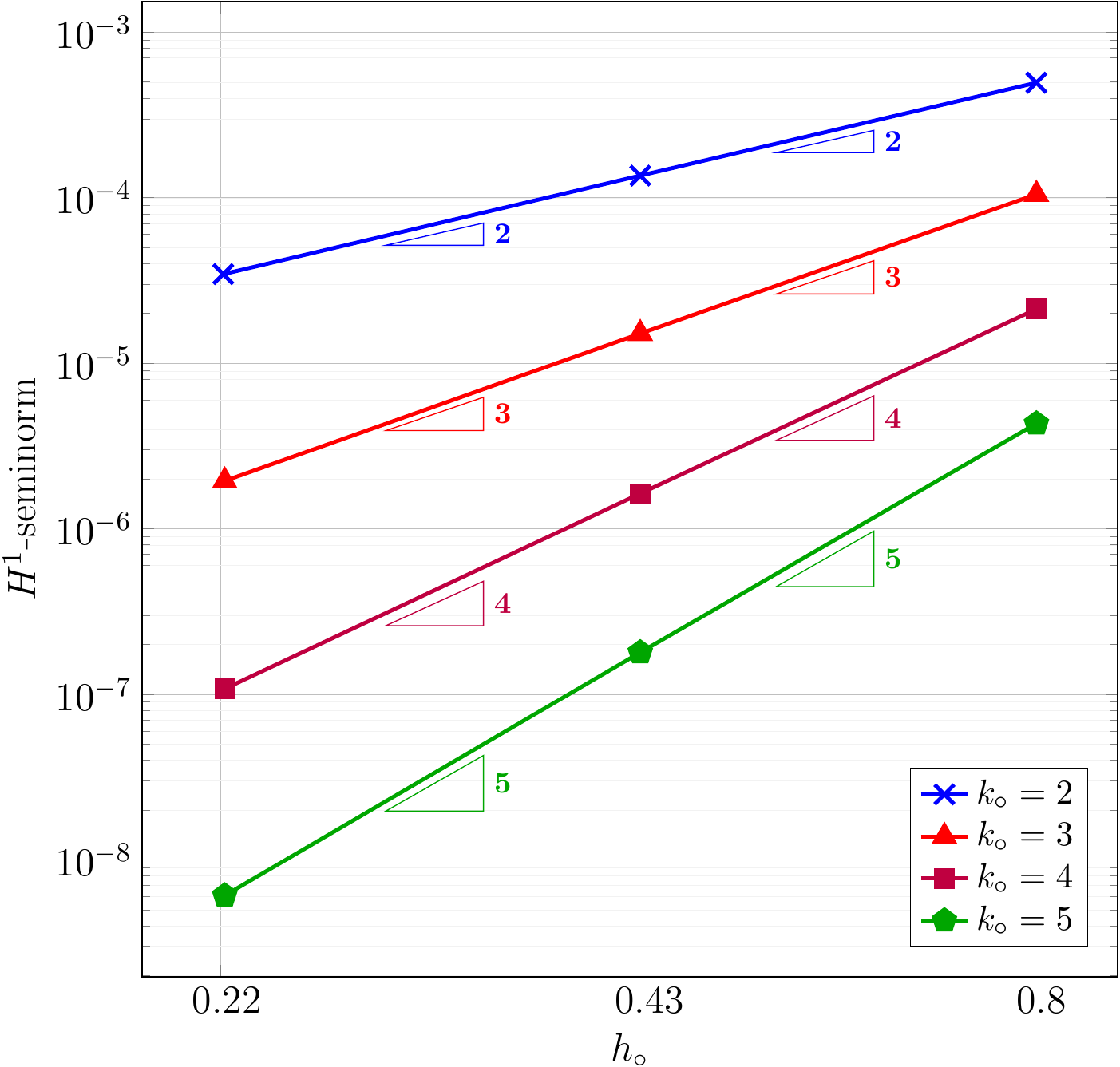}
  \end{minipage}
  \caption{Example 1. Behaviour of the $H^1$-seminorm relative error for $k_\partial = 2$ (left plot) and $k_\partial = 3$ (right plot) by varying $k_\circ$ (lev. 0,1,2).}  \label{fig:H1}
\end{figure}

\begin{figure}[h]
  \centering
  \begin{minipage}[b]{0.42\textwidth}
    \includegraphics[width=\textwidth]{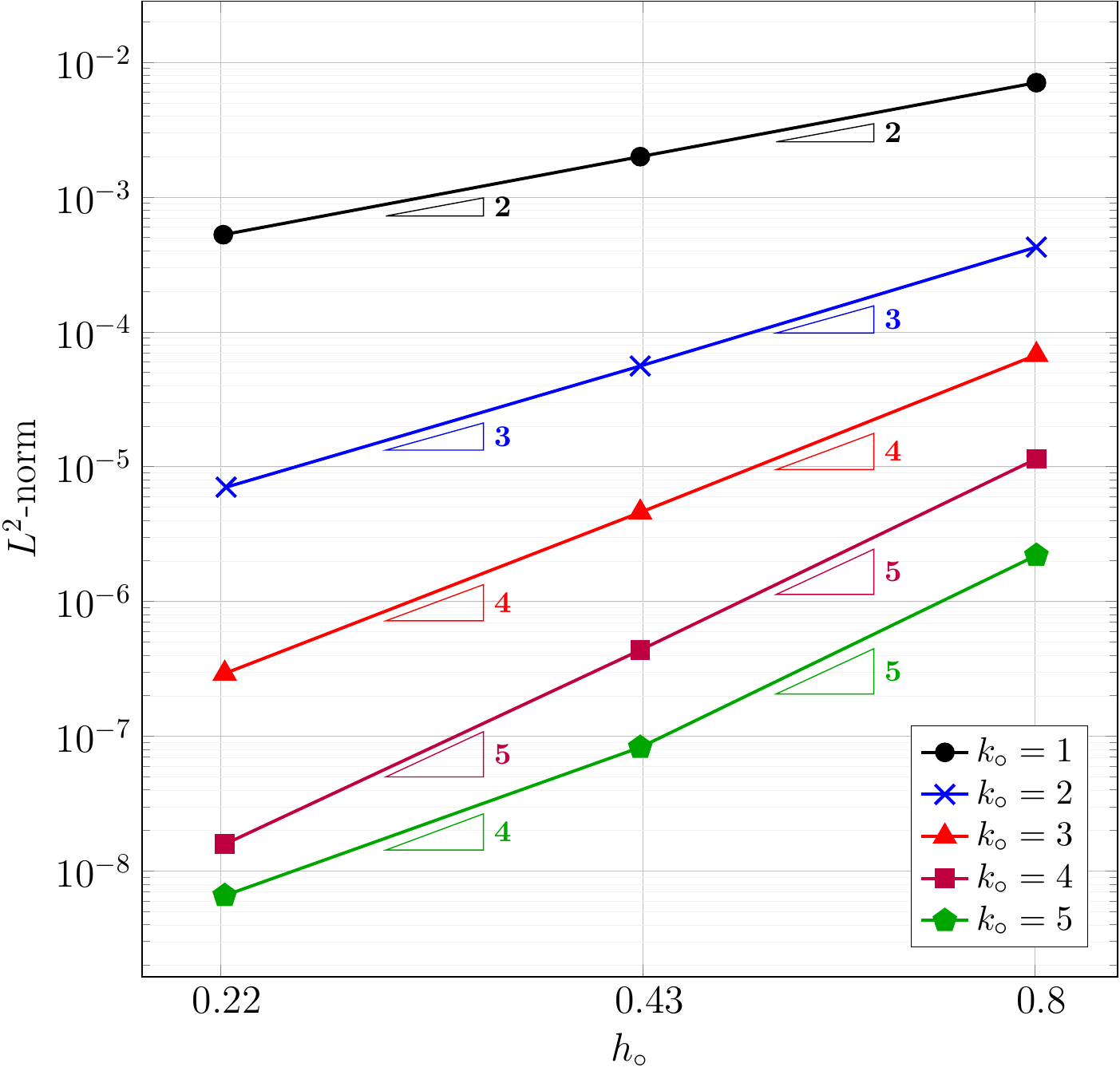}
  \end{minipage}
  \hfill
  \begin{minipage}[b]{0.42\textwidth}
    \includegraphics[width=\textwidth]{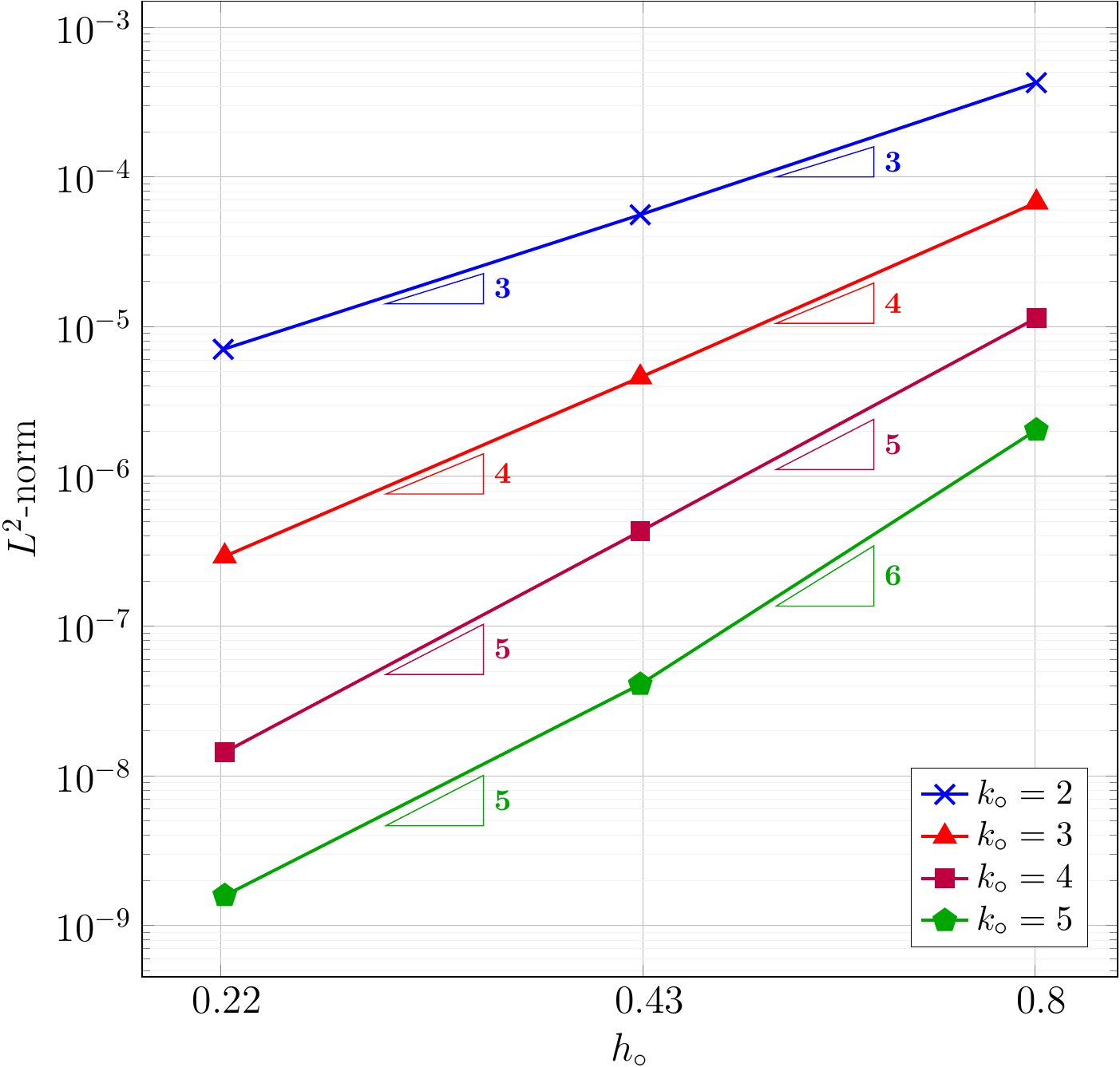}
  \end{minipage}
  \caption{Example 1. Behaviour of the $L^2$-norm relative error for $k_\partial = 2$ (left plot) and $k_\partial = 3$ (right plot) by varying $k_\circ$ (lev. 0,1,2).} \label{fig:L2}
\end{figure}

%%%%%%%%%%%%%%%%%%%%%%%%%%%%%%%%%%%%%%%%%%%%%

\subsection{Example 2} \label{example_2}
We consider the example proposed in \cite{LeRoux1977} (and in \cite{BertoluzzaFalletta2019}), for which $\Gamma_0$ is the boundary of the unit disk, centered at the origin of the cartesian axis, $f=0$ and the datum $g$ on $\Gamma_0$ is defined as
\begin{equation*}
g(\mathbf{x}) = \begin{cases}
x_1^4 & x_1 \ge 0, \\ 0 & x_1 < 0.
\end{cases}
\end{equation*}
Solving the Dirichlet Laplace problem in polar coordinates, and expanding the solution in terms of the eigenvectors of the associated Sturm Liouville system, the solution in polar coordinates reads
$$
u(\rho,\theta) = \frac{3}{16} + \frac{\rho^{-2}}{4} \cos(2\theta) + \frac{\rho^{-4}}{16} \cos(4\theta) + \frac{48}{\pi} \sum_{\underset{n \text{~odd}} {n=1} }^\infty \frac{(-1)^{\nicefrac{(n-1)}{2}} \rho^{-n}}{n^5-20n^3 + 64n} \cos{(n\theta)},
$$
from which we deduce that the asymptotic behaviour is characterized by the constant $\alpha = \nicefrac{3}{16} = 0.1875$. We choose as artificial boundary the ellipse of semi-axes $50$ and $15$, so that the values of the numerical solution at the points $(-50,0)$ and $(50,0)$ can be considered good approximations of $\alpha$. In Figure \ref{fig:asyntotic} we compare the behaviour of the exact and numerical solutions in the intervals $[-50,-1]$ (left plot) and $[1, 50]$ (right plot) for a fixed mesh of the computational domain and for different choices of the approximation orders. Besides noting a very good agreement of the solutions, we report that the corresponding absolute errors at $(-50,0)$ and $(50,0)$ are approximately equal to $4.0e-04$ for $k_\circ = 1$ and $k_\partial = 2$, $5.0e-05$ for $k_\circ = k_\partial = 2$ and $1.0e-08$ for $k_\circ = k_\partial = 3$. 
\begin{figure}[h]
  \centering
  \begin{minipage}[b]{0.45\textwidth}
    \includegraphics[width=\textwidth]{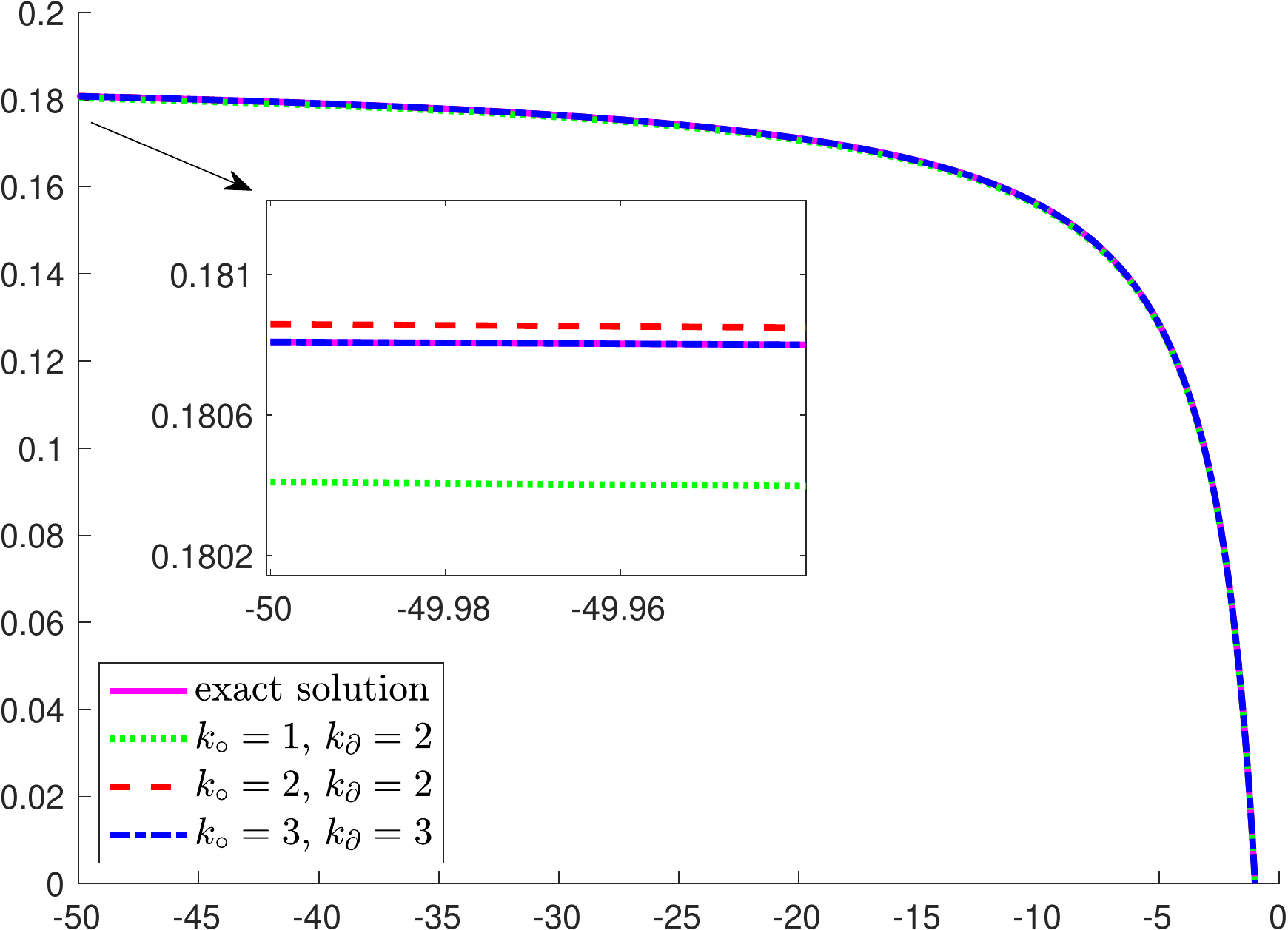}
  \end{minipage}
  \hfill
  \begin{minipage}[b]{0.45\textwidth}
    \includegraphics[width=\textwidth]{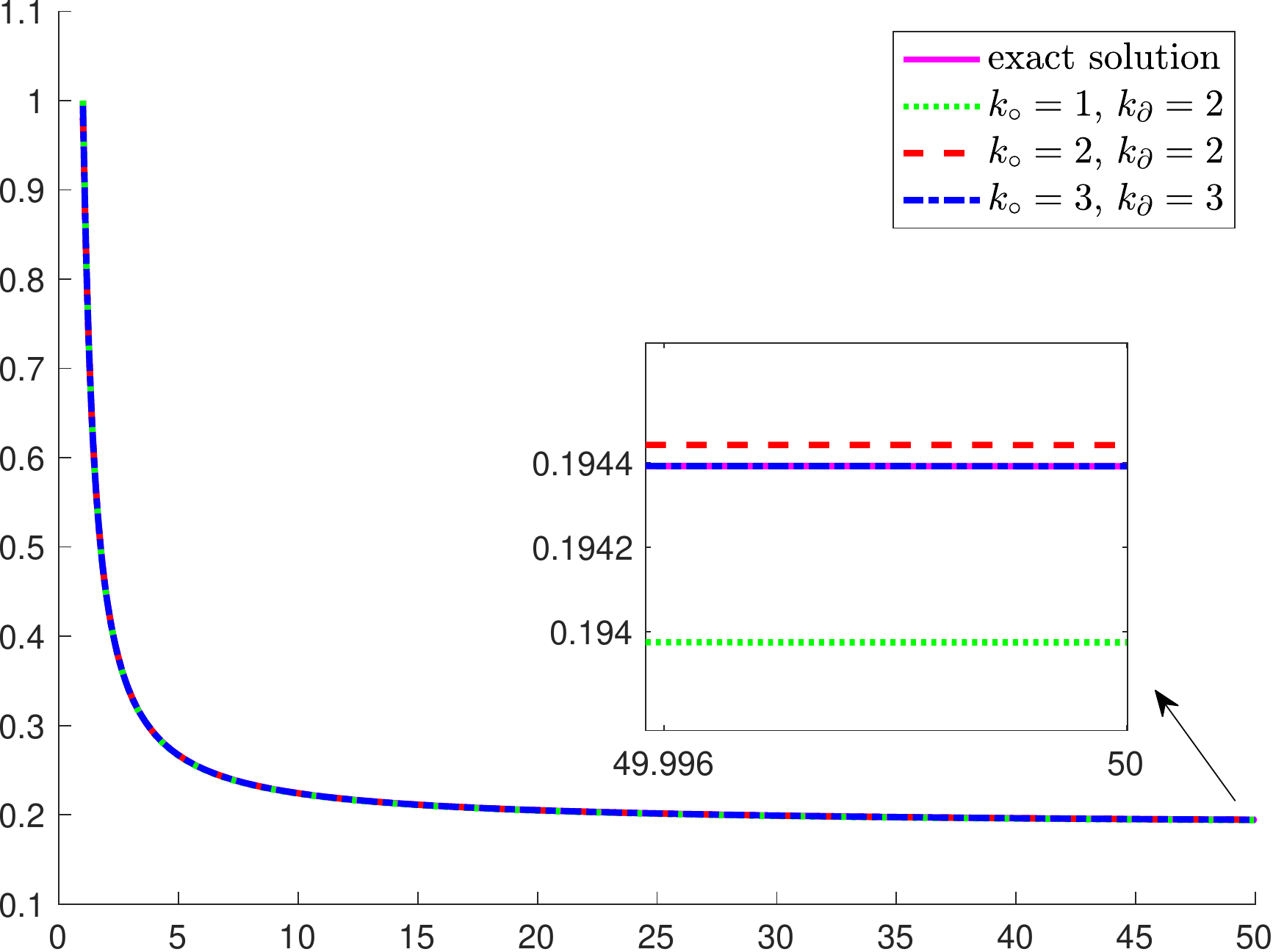}
  \end{minipage}
  \caption{Exact and numerical solutions in $[-50,-1]$ (left plot) and $[1,50]$ (right plot) by varying $k_\circ$ and $k_\partial$.}  \label{fig:asyntotic}
\end{figure}

\section{Conclusions}\label{sec_7_conclusions}
We have proposed and analysed the coupling of a Curved Virtual Element Method with the one-equation Boundary Element Method to solve 2D exterior Poisson problems. The peculiarity of the scheme consists in the use of decoupled approximation orders for the interior CVEM and the boundary integral NRBC. This strategy has allowed us to exploit the well-known flexibility of the CVEM to retrieve an accurate solution by a low order approximation for the BEM.
Since high order BEMs require non-trivial computational efforts to efficiently evaluate the matrix entries of the associated integral operators, the advantage of using a low order BEM turns out to be a key aspect to achieve a good accuracy and convergence rate weighted against computational costs.

The good performances obtained by applying the proposed scheme to elliptic problems, encourage us to consider it within other contexts, such as time dependent exterior problems, for which both the pure BEM and its coupling with standard interior domain methods could become prohibitive.

\section*{Declarations}

This work was performed as part of the GNCS-INdAM 2020 research program \emph{``Metodologie innovative per problemi di propagazione di onde in domini illimitati: aspetti teorici e computazionali''}.

The third author was partially supported by MIUR grant \emph{``Dipartimenti di Eccellenza 2018-2022''}, CUP E11G18000350001.

\bibliographystyle{plain}
\bibliography{bibtex_num}

\end{document}